\documentclass[11pt]{article}

\newcommand{\nats}{\mbox{$\mathbb N$}}
\newcommand{\rats}{\mbox{$\mathbb Q$}}
\newcommand{\ints}{\mbox{$\mathbb Z$}}

\newcommand{\lcm}{{\mathop{\mathrm{lcm}}\nolimits}}
\newcommand{\Gal}{{\mathop{\mathrm{Gal}}\nolimits}}
\newcommand{\EC}{{\mathop{\mathrm{EC}}\nolimits}}
\newcommand{\ES}{{\mathop{\mathrm{ES}}\nolimits}}
\newcommand{\cs}{{\mathop{\mathrm{cs}}\nolimits}}
\newcommand{\comment}[1]{}

\usepackage{amsmath}  
\usepackage{amssymb}
\usepackage{graphicx}                   
\usepackage{stackrel}		

\def\squarebox#1{\hbox to #1{\hfill\vbox to #1{\vfill}}}
\def\qed{\hspace*{\fill}
        \vbox{\hrule\hbox{\vrule\squarebox{.667em}\vrule}\hrule}\smallskip}
\newenvironment{proof}{\begin{trivlist}
  \item[\hspace{\labelsep}{\em\noindent Proof.~}]
  }{\qed\end{trivlist}}

\newtheorem{lemma}{Lemma}[section]
\newtheorem{theorem}[lemma]{Theorem}
\newtheorem{corollary}[lemma]{Corollary}
\newtheorem{proposition}[lemma]{Proposition}
\newtheorem{claim}[lemma]{Claim}
\newtheorem{observation}[lemma]{Observation}
\newtheorem{definition}[lemma]{Definition}

\def\squareforqed{\hbox{\rlap{$\sqcap$}$\sqcup$}}
\def\qed{\ifmmode\squareforqed\else{\unskip\nobreak\hfil
\penalty50\hskip1em\null\nobreak\hfil\squareforqed
\parfillskip=0pt\finalhyphendemerits=0\endgraf}\fi}

\newlength{\tablength}
\newlength{\spacelength}
\newcommand{\tabstar}{\hspace*{\tablength}}
\newcommand{\spacestar}{\hspace*{\spacelength}}
\def\obeytabs{\catcode`\^^I=\active}
{\obeytabs\global\let^^I=\tabstar}
{\obeyspaces\global\let =\spacestar}
\newenvironment{display}{\begingroup\obeylines\obeyspaces\obeytabs}{\endgroup}
\newenvironment{prog}{\begin{display}\parskip0pt\sf}{\end{display}}

\title{On integer radii coin representations of the wheel graph}

\author{
{\sl Geir Agnarsson}
\thanks{Department of Mathematical Sciences,
George Mason University,
MS 3F2,
4400 University Drive,
Fairfax, VA -- 22030, USA,
{\tt geir@math.gmu.edu}}
\and
{\sl Jill Bigley Dunham}
\thanks{Department of Mathematics,
Hood College, 401 Rosemont Avenue, Frederick, MD -- 21701, USA,
{\tt dunham@hood.edu}} 
}

\date{}

\begin{document}

\maketitle

\begin{abstract}
A {\em flower} is a coin graph representation
of the wheel graph. A {\em petal} of the wheel graph is an edge
to the center vertex. In this paper we investigate flowers whose
coins have integer radii. For an $n$-petaled flower we
show there is a unique irreducible polynomial $P_n$ in $n$ variables 
over the integers $\ints$, the affine variety of which contains
the cosines of the internal angles formed by the petals of the flower.
We also establish a recursion that these irreducible polynomials
satisfy. Using the polynomials $P_n$, we develop a parameterization 
for all the integer radii of the coins of the 3-petal flower.

\vspace{3 mm}


\vspace{2 mm}

\end{abstract}

\section{Introduction}
\label{sec:intro}

By a {\em coin graph} we mean a graph whose vertices can be 
represented as closed, non-overlapping disks in the Euclidean 
plane such that two vertices are adjacent
if and only if their corresponding disks intersect at
their boundaries, i.e.~they touch. For $n\in\nats$ the {\em wheel graph}
$W_n$ on $n+1$ vertices is the simple graph obtained by connecting an
additional center vertex to all the vertices of the cycle $C_n$ on $n$ vertices.
These additional edges are called {\em petals}. A coin graph representation
of a wheel graph is called a {\em flower}. In Figure~\ref{fig:flower-not}
we see an example of a flower on the left, and a configuration
of coins that does not form a flower on the right.
\begin{figure}[tb]
\begin{center}
\includegraphics[scale=1]{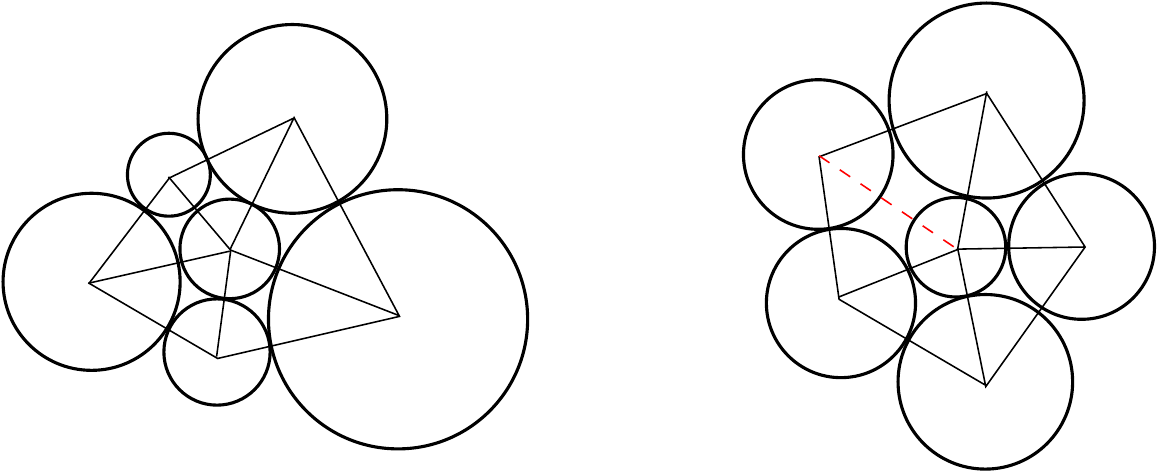}
\caption{Examples of a flower and a non-flower.}
\label{fig:flower-not}
\end{center}
\end{figure}

The study of flowers is central in many discrete geometrical settings,
in particular in circle packings~\cite{Stephenson} and also in the study
of planar graphs in general, since every planar graph has a coin
graph representation. That a coin graph is planar is clear, but
that the converse is true is a nontrivial topological result, usually
credited to Thurston~\cite{Thurston}, but is also due to
both Koebe~\cite{Koebe} and Andreev~\cite{Andreev}. For a brief history of this
result we refer to~\cite[p.~118]{Ziegler}. Numerous simply stated,
but extremely hard problems involving coin graphs can be found
in a recent and excellent collection of research problems
in discrete geometry~\cite{Brass}. Also, 
Brightwell and Scheinerman~\cite{Scheinerman} 
explored integral representations of coin graphs, where the radii of the 
coins can take arbitrary positive integer values.

In this paper we study algebraic relations the radii of flowers
must satisfy. We first show that for every $n\geq 3$ the cosines 
of the central angles of an $n$-petal flower are contained
in the affine variety of an irreducible polynomial $P_n$ 
in $n$ variables over the integers. We note that the cosines
are more interesting than the sines in this case, for the mere
reason that cosines of the angles of an integer-sided triangle
are all rational. In particular, for the case $n=3$
we then find a parametrization of all integer 
$n$-tuples in this variety of $P_n$. 
Also for the case of $n=3$,
we obtain all rational, and hence integer, radii of four
mutually tangent circles, sometimes called {\em Soddy circles}
as Frederick Soddy rediscovered Descartes' Circle Theorem
in 1936~\cite{Austin}. Our parametrization differs from 
the one obtained by Graham et al.~in~\cite{Graham} as it
is free of any equations relating the parameters.

The rest of the paper is organized as follows: in Section~\ref{sec:setup}
we state our main terminology and definitions.
We also present and discuss some basic observations and 
consequences from the definitions. In Section~\ref{sec:irr-poly} 
we use Galois theory to formally define the polynomials
$P_n(x_1,\ldots,x_n)$ whose affine variety contains
$(\cos\theta_1,\ldots,\cos\theta_n)$ where 
$\theta_1,\ldots,\theta_n$ are the internal angles
of an $n$-petaled flower. We then prove our main result
of this paper, that each $P_n$ is an irreducible polynomial
over $\rats$. In Section~\ref{sec:n=3} we consider the special
case of a 3-petal flower. In this case we have
four mutually tangent Soddy circles, and we derive a free
parametrization of all rational radii of the outer circles
when the inner circle has radius one. This will then yield
an equation free parametrization of all integer radii
of four mutually tangent Soddy circles. 

\section{Definitions, setup and basic informal observations}
\label{sec:setup}
 
In what follows $\nats = \{1,2,3,\ldots\}$ is the set of natural numbers.
For $n\in\nats$ we let $[n] = \{1,\ldots,n\}$. 
For each $n\in\nats$, an $n$-petal flower imposes a relation on the radii 
of its coins. For such a flower, asssume the radius of the 
center coin is $r$ and the radii of the $n$ outer coins are 
$r_1, \ldots, r_n$ in clockwise order. We first note
that there is an obvious equation relating the $r_i$:
for each pair of radii $r_i$ and $r_{i+1}$ of consecutive petals 
around a center coin of radius $r$ we obtain a triangle with 
sides of length $r+r_i$, $r+r_{i+1}$, and  $r_i + r_{i+1}$ and the 
angle $\theta_i$ at the center vertex is given by
\begin{equation}
\label{eqn:coslaw}
\theta_i = \arccos \left(\frac{(r+r_i)^2 + (r+r_{i+1})^2 - 
(r_i + r_{i+1})^2}{2(r + r_i)(r+r_{i+1})}\right).
\end{equation}
The equation that determines a flower with petals of radii 
$r_1, \ldots, r_n$ is
\begin{equation}
\label{eqn:2pi}
\sum_{i=1}^n \theta_i = 2\pi.
\end{equation}
For $G \subseteq S_n$, a polynomial $f$ is {\em $G$-symmetric} if 
$f(x_1, \dots, x_n) = f(x_{\sigma(1)}, \dots, x_{\sigma(n)})$ for all 
$\sigma \in G$. We see that (\ref{eqn:2pi}) is a $D_n$-symmetric 
function in terms of $r_1/r, \ldots, r_n/r$, where $D_n$ is the 
dihedral group of symmetries on the regular polygon with $n$ sides.  
In~\cite{Stanley} it is shown that for {\em reflection groups} like the 
dihedral group $D_n$ there is a basis of polynomials just like the 
elementary symmetric functions for the symmetric group $S_n$.
As we will discuss, if $x_i = \cos\theta_i$ for each $i \in [n]$, then
(\ref{eqn:2pi}) will corresponds to a symmetric polynomial 
$f\in\mathbb{Q}[x_1,\ldots, x_n]$. Also, 
if the center coin has radius $r=1$, and so 
$\theta_i = \theta_i(1, r_i, r_{i+1})$ is a function of only the
two consequtive radii $r_i$ and $r_{i+1}$, then (\ref{eqn:2pi}) 
will corresponds to a $D_n$-symmetric polynomial 
$g\in \mathbb{Q}[r_1,\ldots,r_k]$. In particular, for general 
radius $r$ of the center vertex (replacing $r_i$ with $r_i/r$), 
if $d= \deg(g)$, which we define as the sum degree, then 
$r^d g\left(\frac{r_1}{r}, \ldots, \frac{r_n}{r}\right) \in 
\mathbb{Q}[r, r_1, \ldots, r_n]$ is a homogeneous element and
\begin{equation}
\label{eqn:g}
r^d g\left(\frac{r_1}{r}, \ldots, \frac{r_k}{r}\right) = 
\sum_{i=0}^d g_i r^i \in \mathbb{Q}[r_1, \ldots, r_k][r],
\end{equation}
where each $g_i \in \mathbb{Q}[r_1, \ldots, r_n]$ is a $D_n$-symmetric 
polynomial. Although intuitively clear, we will in what follows 
demonstrate this claim informally in an explicit example.  
To obtain a symmetric function $f=f(x_1, \ldots, x_n)$ we will take the 
cosine of both sides of (\ref{eqn:2pi}). Using the relation 
$e^{i\theta} = \cos\theta+ i \sin\theta$ and then taking the real
and imaginary parts of 
$e^{i(\theta_1 + \ldots + \theta_n)}= e^{i \theta_1} \cdots e^{i \theta_n}$,
we obtain the following technical lemma.
\begin{lemma}
\label{lmm:sum}
For $n \geq 1$ we have the following generalized addition formulae 
for $\cos$ and $\sin$:
\[
\cos \left( \sum_{i=1}^n \theta_i \right) =  
\sum_{2^{n-1} \ terms} \pm \cs(\theta_1)\cs(\theta_2)\cdots \cs(\theta_n),
\]
where the sum on the right is taken over the $2^{n-1}$ possible terms where 
(i) each $\cs$-function represents either $\sin$ or $\cos$ and 
(ii) each term has an even number $2e$ of $\sin$-functions and the 
sign of the term is given by $(-1)^e$.

Similarly for $\sin$ we have
\begin{eqnarray*}
\sin \left( \sum_{i=1}^n \theta_i \right)	& = & 
\sum_{2^{n-1} \ terms} \pm \cs(\theta_1)\cs(\theta_2)\cdots \cs(\theta_n),
\end{eqnarray*}
where the sum on the right is taken over the $2^{n-1}$ possible terms where 
(i) each $\cs$-function represents either $\sin$ or $\cos$ and 
(ii) each term has an odd number $2e+1$ of $\sin$-functions and the sign of 
the term is given by $(-1)^e$.
\end{lemma}
If $x_i=\cos \theta_i$ for each $i\in[n]$ then $y_i= \sin \theta_i$ 
satisfies the equation $x_i^2 + y_i^2 = 1$ and hence 
$y_i= \pm \sqrt{1-x_i^2}$.  The geometric properties of the coin 
graph determine that for the interior angles $\theta_i$ we have 
$0\leq \theta_i < \pi$ 
and so $\sin \theta_i \geq 0$. Hence we have $y_i= \sqrt{1-x_i^2}$
and so both $\cos\theta_i$ and $\sin\theta_i$ are in terms of $x_i$.
\begin{definition}
\label{def:ECES}
We define the algebraic expressions $\EC_n$ and $\ES_n$ 
by taking the sine or cosine of (\ref{eqn:2pi}) and expanding
using Lemma~\ref{lmm:sum}.
\[
\EC_n(x_1, \ldots, x_n)	= \cos \left( \sum_{i=1}^n \theta_i \right), \ \ 
\ES_n(x_1, \ldots, x_n)	= \sin \left( \sum_{i=1}^n \theta_i \right).
\]
\end{definition}
{\sc Example:}
For $n=1$ we have
\[
\EC_1(x_1)= x_1, \ \ \ES_1(x_1) =  y_1 = \sqrt{1-x_1^2},
\]
and for $n=2$ we have 
\[
\EC_2(x_1, x_2)	= x_1 x_2 - \sqrt{1-x_1^2} \sqrt{1-x_2^2},
\ \ 
\ES_2(x_1, x_2)	= x_2\sqrt{1-x_1^2}  + x_1 \sqrt{1-x_2^2}.
\]
Directly by the addition formulae for cosine and sine we have
the following recursive property of these expressions.
\begin{lemma}
\label{lmm:ECES-rec}
For each $i \in \{1, \ldots, n\}$ we have
\begin{eqnarray*}
\EC_n(x_1, \ldots, x_n)	& = &
x_i \EC_{n-1}(\widehat{x_i}) - y_i \ES_{n-1}(\widehat{x_i}), \\ 
\ES_n(x_1, \ldots, x_n)	& = &
y_i \EC_{n-1}(\widehat{x_i}) + x_i \ES_{n-1}(\widehat{x_i}).
\end{eqnarray*}
where $y_i= \sqrt{1-x_i^2}$ and 
$\left(\widehat{x_i}\right) = 
\left(x_1,\ldots, x_{i-1}, x_{i+1}, \ldots, x_n\right)$.  
In particular for $i=1$ we have
\begin{eqnarray*}
\EC_n(x_1, \ldots, x_n)	& = & 
x_1 \EC_{n-1}(\widehat{x_1}) - y_1 \ES_{n-1}(\widehat{x_1}), \\
\ES_n(x_1, \ldots, x_n)	& = &
y_1 \EC_{n-1}(\widehat{x_1}) + x_1 \ES_{n-1}(\widehat{x_1}).
\end{eqnarray*}
\end{lemma}
Note that the expressions $\EC_n$ and $\ES_n$ 
are symmetric in $x_1,\dots,x_n$. As informally demonstrated here below,
these will yield symmetric
polynomials (see~\cite[p.~252]{Hungerford} for more information and general 
algebraic properties of symmetric polynomials.)

For a fixed $n\in\nats$ (\ref{eqn:2pi}) yields 
the algebraic equation $\EC_n = 1$. By repeatedly isolating one
term that contains a $y_i$ and squaring, then rearranging the terms,
we obtain a polynomial equation $C_n = 0$. For example for 
$n = 1,2,3,4$ we obtain
\begin{eqnarray*}
C_1(x_1)	& = & x_1 -1\\
C_2(x_1, x_2)   & = & (x_1-x_2)^2\\
C_3(x_1, x_2, x_3)  & = & (x_1^2+x_2^2+x_3^2 -2x_1x_2x_3-1)^2\\
C_4(x_1, x_2, x_3, x_4) & = & (x_1^4+x_2^4+x_3^4+x_4^4 \\
& & - 
2(x_1^2x_2^2+x_2^2x_3^2 + x_3^2x_4^2 + x_1^2x_4^2 + x_1^2x_3^2 + x_2^2x_4^2) \\
& & + 4(x_1^2x_2^2x_3^2 + x_2^2x_3^2x_4^2 + x_1^2x_3^2x_4^2 + x_1^2x_2^2x_4^2) \\
& & + 4x_1x_2x_3x_4(2-x_1^2-x_2^2-x_3^2-x_4^2))^2.
\end{eqnarray*}
Note that for $n\geq 2$ it appears that $C_n$ is always
a square polynomial, something we will prove 
in Section~\ref{sec:irr-poly}.

By (\ref{eqn:coslaw}) we have for each $i\in[n]$
\[
x_i = \frac{(r+r_i)^2 + (r+r_{i+1})^2 
- (r_i + r_{i+1})^2}{2(r + r_i)(r+r_{i+1})}.
\]
Substituting these $x_i$ into the polynomial equation $C_n=0$ 
yields a rational equation in $r_1, \ldots, r_n$ and $r$.  
This rational equation can then be transformed into a polynomial
equation $g=0$ where $g \in \mathbb{Q}[r_1, \ldots, r_k]$ as 
in (\ref{eqn:g}). That the polynomial will be $D_n$-symmetric is 
clear from geometry: it does not matter 
which angle we label $\theta_1$ (rotation) or whether we do our 
numbering clockwise or counter-clockwise (reflection.)

{\sc Example:}
For $n=3$ we have the equation 
$f=C_3= (x_1^2+x_2^2+x_3^2 -2x_1x_2x_3-1)^2 = 0$ and by substitution
of the radii into the equation obtain
\begin{eqnarray*}
C_3	
& = & \left(\left(\frac{(r+r_1)^2 + (r+r_{2})^2 - 
(r_1 + r_{2})^2}{2(r + r_1)(r+r_{2})}\right)^2 + 
\left(\frac{(r+r_2)^2 + (r+r_{3})^2 - 
(r_2 + r_{3})^2}{2(r + r_2)(r+r_{3})}\right)^2 \right.\nonumber\\
&   & + \left.\nonumber \left(\frac{(r+r_3)^2 + 
(r+r_{1})^2 - (r_3 + r_{1})^2}{2(r + r_3)(r+r_{1})}\right)^2 - 
2\left(\frac{(r+r_1)^2 + (r+r_{2})^2 - 
(r_1 + r_{2})^2}{2(r + r_1)(r+r_{2})}\right) \right.\nonumber\\
& & \left.\nonumber \left(\frac{(r+r_2)^2 + (r+r_{3})^2 - 
(r_2 + r_{3})^2}{2(r + r_2)(r+r_{3})}\right)
\left(\frac{(r+r_3)^2 + (r+r_{1})^2 - 
(r_3 + r_{1})^2}{2(r + r_3)(r+r_{1})}\right) 
- 1\right)^2\\
& = & \frac{16}{(r+r_1)^4 (r+r_2)^4 (r+r_3)^4}
\left(-2 r_1^2 r_2 r_3 r^2+r_1^2  r_3^2r^2+r_2^2 r_3^2 r^2+
r_1^2 r_2^2 r^2-2  r_1 r_2 r_3^2r^2\right.\nonumber\\
&   & - \left.\nonumber 2 r_1 r_2^2 r_3 r^2-
2 r_1^2 r_2^2 r_3 r-2 r_1 r_2^2 r_3^2 r-2 r_1^2 r_2 r_3^2 r+
r_1^2 r_2^2 r_3^2\right)^2\\
	& = & 0.
\end{eqnarray*}
Hence, our polynomial $g$ is then given by
\begin{eqnarray*}
g	
& = & 16\left(-2 r_1^2 r_2 r_3 r^2+r_1^2 r_3^2 r^2 +r_2^2 r_3^2 r^2+
r_1^2 r_2^2 r^2-2 r_1 r_2 r_3^2 r^2\right.\\
&   & \left.- 2 r_1 r_2^2 r_3 r^2-
2 r_1^2 r_2^2 r_3 r-2  r_1 r_2^2r_3^2 r-2 r_1^2 r_2 r_3^2 r+
r_1^2 r_2^2 r_3^2\right)^2.\\
\end{eqnarray*}
We now write
\[
r^{12} g\left(\frac{r_1}{r}, \frac{r_2}{r}, \frac{r_3}{r}\right)	
= r^8 g_8 +  r^6g_6 +  r^4g_4  - r^2g_2 +g_0 \in\mathbb{Q}[r_1, r_2, r_3][r],
\]
where 
\begin{eqnarray*}
g_8 & = & 
16(r_2^4  r_3^4 +r_1^4  r_2^4+r_1^4  r_3^4+
4 r_1^3  r_2^3  r_3^2-4  r_1^4  r_2  r_3^3 + 
4  r_1^2  r_2^3  r_3^3-4  r_1^4  r_2^3  r_3 
 -  4  r_1^3  r_2^4  r_3\\
 &   & 
+ 6  r_1^2 r_2^2 r_3^4  -4  r_1^3  r_2 r_3^4 - 4  r_1r_2^3  r_3^4 -4  
r_1r_2^4  r_3^3 +6  r_1^4  r_2^2  r_3^2
 +  6 r_1^2 r_2^4  r_3^2  +4  r_1^3  r_2^2  r_3^3),\\
g_6 & = & 
64(r_1^3 r_2^2 r_3^4  - r_1^4   r_2 r_3^4+6 r_1^3  r_2^3  r_3^3+ 
r_1^4  r_2^2  r_3^3+r_1^4  r_2^3  r_3^2
 +  r_1^3  r_2^4  r_3^2+r_1^2 r_2^3  r_3^4  + r_1^2r_2^4  r_3^3\\  
&  & - r_1 r_2^4  r_3^4  -r_1^4  r_2^4  r_3),\\
g_4 & = & 
32(3  r_1^4 r_2^2 r_3^4   +2  r_1^4  r_2^3  r_3^3+2  r_1^3  r_2^3  r_3^4 +
3  r_1^4  r_2^4  r_3^2+3 r_1^2 r_2^4  r_3^4
 +  2r_1^3 r_2^4   r_3^3),\\
g_2 & = & 
64(r_1^3 r_2^4  r_3^4-r_1^4  r_2^3  r_3^4- r_1^4  r_2^4  r_3^3),\\
g_0 & = & 16  r_1^4  r_2^4  r_3^4,
\end{eqnarray*}
and each of these $g_i \in \mathbb{Q}[r_1, r_2, r_3]$ is a $D_3$-symmetric 
polynomial.

\vspace{3 mm}

In general, for the terms with degree of $\delta\in\{0,\ldots,d\}$,
then $r^d$ will 
cancel out  all the denominators and we will be left with a term
$r^{d-\delta}g_{d-\delta}$ where $g_{d-\delta}$ is an 
element of $\mathbb{Q}[r_1, ..., r_n]$. That $g_{d-\delta}$ will be 
$D_n$-symmetric follows from the $D_n$-symmetry of $g$ 
and hence, viewing $g$ as a polynomial in $r$ alone, each coefficient
for each power of $r$ is also $D_n$-symmetric.

\section{The polynomial of the $n$-petal flower and its irreducibility}
\label{sec:irr-poly}

This section forms the main contribution and results of the paper.
We will show that for each $n \geq 2$ we have $C_n = P_n^2$, 
where $P_n$ is an irreducible polynomial for $n \geq 2$, 
and $P_n$ is symmetric for $n \geq 3$. To proceed we need some
preliminary definitions and results.
\begin{definition}
\label{def:galois}
For $n\in\nats$ let $G^*_n$ be the Galois group of automorphisms
on $\mathbb{Q}(x_1, \ldots, x_n, y_1, \ldots, y_n)$ that fixes the field 
$\mathbb{Q}(x_1, \ldots, x_n)$. Also, let $G_n$ be the Galois group of automorphisms on 
$\mathbb{Q}(x_1, \ldots, x_n, y_iy_j : i < j)$ that fixes the field 
$\mathbb{Q}(x_1, \ldots, x_n)$. That is,
\begin{eqnarray*}
G^*_n & = & \Gal(\mathbb{Q}(x_1, x_2, \ldots, x_n, y_1, 
\ldots, y_n)/\mathbb{Q}(x_1, \ldots, x_n)),\\
G_n   & = & \Gal(\mathbb{Q}(x_1, x_2, \ldots, 
x_n, y_iy_j : i < j)/\mathbb{Q}(x_1, \ldots, x_n)),
\end{eqnarray*}
where $x_1,\ldots,x_n$ are algebraically independent indeterminates and 
$y_i = \sqrt{1-x_i^2}$ for each $i$, that is $y_i$ is one root of
$X^2 + x_i^2 - 1 = 0 \in \mathbb{Q}(x_1, \ldots, x_n)[X]$.
\end{definition}
\begin{lemma}
\label{lmm:ZnZn-1}
For $n\geq 1$ we have $G_n^* \cong \mathbb{Z}_2^n$ and 
$G_n \cong \mathbb{Z}_2^{n-1}$.
\end{lemma}
\begin{proof}
For $G_n^*$, each $y_i$ is the root of an irreducible quadratic polynomial 
$X^2 - (1-x_i^2)$ from the ring 
$\mathbb{Q}(x_1, \ldots, x_n, y_1, \ldots, y_{i-1})[X]$,
which is the minimum polynomial of $y_i$ 
for each $i$. Hence we have $G_n^* \cong \mathbb{Z}_2^n$.

For $G_n$, each $y_iy_j$ with $i < j$ is also the root of an irreducible 
quadratic polynomial 
$X^2 - (1-x_i^2)(1-x_j^2) \in \mathbb{Q}(x_1, \ldots, x_n)[X]$.  
However, every element of $\mathbb{Q}(x_1, x_2, \ldots, x_n, y_iy_j : i < j)$ 
can be written as a rational function in terms of only elements of the form 
$y_i y_{i+1}$ as follows:
\[
y_i y_j	= \frac{(y_i y_{i+1}) (y_{i+1}y_{i+2}) 
\cdots (y_{j-1} y_j)}{y_{i+1}^2 \cdots y_{j-1}^2} =  
\frac{(y_i y_{i+1}) (y_{i+1}y_{i+2}) \cdots 
(y_{j-1} y_j)}{(1-x_{i+1}^2) \cdots (1-x_{j-1}^2)}.
\]
So we have that 
\[
\mathbb{Q}(x_1, x_2, \ldots, x_n, y_iy_j : i < j) = 
\mathbb{Q}(x_1, x_2, \ldots, x_n, y_iy_{i+1}: 1 \leq i <n ).
\]
Each term $y_iy_{i+1}$ is a root of an irreducible 
quadratic polynomial $X^2 - (1-x_i^2)(1-x_{i+1}^2)$ from the
ring $\mathbb{Q}(x_1, \ldots, x_n, y_1y_2, \ldots, y_{i-1}y_i)[X]$, 
which is the minimal polynomial of $y_iy_{i+1}$ 
for each $i \in \{1, \ldots, n-1\}$.  Therefore we have that 
$G_n \cong \mathbb{Z}_2^{n-1}$.
\end{proof}
\begin{lemma}
\label{lmm:neg}
For $n \in \mathbb{N}$, the group $G_n \cong \mathbb{Z}_2^{n-1}$ 
can be presented as
\[
G_n = \left\langle \sigma_1, \ldots, \sigma_{n-1} : 
\sigma_i^2 = e, \sigma_i \sigma_j = \sigma_j \sigma_i \right\rangle,
\]
where each $\sigma_i$ is an automorphism fixing 
$\mathbb{Q}(x_1, \ldots x_n)$ and
\[
\sigma_i(y_j y_{j+1}) =  \left\{ \begin{array}{lll}
                                   -y_j y_{j+1}    & \mbox{if $i=j$} \\
                                    y_j y_{j+1}    & \mbox{if $i \neq j$}.
                                 \end{array}
                         \right.
\]
\end{lemma}
\begin{proof}
Since $(y_i y_{i+1})^2 = (1-x_i^2) (1-x_{i+1}^2)$ and the Galois group 
$G_n$ is fixing the $x_i$, the only possible automorphisms are 
$\sigma(y_i y_{i+1}) = -y_i y_{i+1}$ and $\sigma(y_i y_{i+1}) = y_i y_{i+1}$.  
We can then generate the group as in the statement of the theorem with 
$n-1$ generators $\sigma_i$.
\end{proof}
\begin{corollary}
\label{cor:sigma_signs}
For every $\sigma \in G_n$, let $s_{\sigma; j} \in \{-1, 1\}$ be such 
that $\sigma(y_j y_{j+1})= s_{\sigma; j} y_j y_{j+1}$.  Then for every 
$i < j$ we have $\sigma(y_i y_{j}) = 
s_{\sigma; i} s_{\sigma; i+1} \cdots s_{\sigma; j} y_i y_{j}$.
In particular, if $i <n$ then $\sigma_{n-1}(y_i y_n) = -y_i y_n$ 
and if $i >1$ then $\sigma_{1}(y_1 y_i) = -y_1 y_i$.
\end{corollary}
We are now able to give a precise definition of $C_n$ from
Section~\ref{sec:setup} for each $n \in \mathbb{N}$.
\begin{definition}
\label{def:cn}
For $n\in\nats$, define the polynomial $C_n\in\mathbb{Q}[x_1,\ldots,x_n]$ 
by
\[
C_n(x_1, \ldots, x_n) = \prod_{\sigma \in G_n} (\sigma(\EC_n) -1).  
\]
\end{definition}
From Definition~\ref{def:cn} we see that $C_n$ is indeed symmetric in 
$x_1,\ldots,x_n$.

{\sc Example:} For $n=2$ we have
$G_2 = \left\langle \sigma \right\rangle\cong \mathbb{Z}_2$ 
where $\sigma(y_1 y_2)= -y_1 y_2$ and $\sigma^2=e$, and hence
\[
\prod_{\sigma \in G_2} (\sigma(\EC_2) -1) =  
(x_1x_2 - y_1y_2 -1)(x_1x_2-\sigma(y_1y_2) - 1) = 
(x_1-x_2)^2 = C_2(x_1, x_2).
\]
For $n=3$, we have $G_3= \left\langle \sigma_1, \sigma_2\right\rangle$ 
where $\sigma_1(y_1 y_2)= -y_1 y_2$ and $\sigma_2(y_2 y_3)= -y_2 y_3$
and hence 
\begin{eqnarray*}
\prod_{\sigma \in G_3} (\sigma(\EC_3) -1) 
& = & (x_1 x_2 x_3 - x_1 y_2 y_3 - y_1 x_2 y_3 - y_1 y_2 x_3-1)\\
& & \cdot(x_1 x_2 x_3 - x_1 y_2 y_3 - \sigma_1(y_1 x_2 y_3) 
- \sigma_1(y_1 y_2 x_3)-1)\\
& & \cdot(x_1 x_2 x_3 - \sigma_2(x_1 y_2 y_3) 
- \sigma_2(y_1 x_2 y_3) - y_1 y_2 x_3-1)\\
& & 
\cdot(x_1 x_2 x_3 - \sigma_1\sigma_2(x_1 y_2 y_3) - 
\sigma_1\sigma_2(y_1 x_2 y_3) - \sigma_1\sigma_2(y_1 y_2 x_3)-1)\\
& = & 
(x_1 x_2 x_3 - x_1 y_2 y_3 - y_1 x_2 y_3 - y_1 y_2 x_3-1)\\
& & \cdot(x_1 x_2 x_3 - x_1 y_2 y_3 + y_1 x_2 y_3 + y_1 y_2 x_3-1)\\
& & \cdot(x_1 x_2 x_3 + x_1 y_2 y_3 + y_1 x_2 y_3 - y_1 y_2 x_3-1)\\
& & \cdot(x_1 x_2 x_3 - x_1 y_2 y_3 + y_1 x_2 y_3 - y_1 y_2 x_3-1)\\
& = & 
(x_1^2+x_2^2+x_3^2 - 2x_1 x_2 x_3 -1)^2.
\end{eqnarray*}
By Lemma~\ref{lmm:sum}, each of the $2^{n-2}$ terms of $\ES_{n-1}$ in 
terms of $x_1, \ldots, x_{n-1}, y_1, \ldots, y_{n-1}$ contains positive 
odd factors of $y_i$ for $i \leq n-1$.  Hence $\sigma_{n-1} \in G_n$ 
fixes $\mathbb{Q}(x_1, \ldots, x_n, y_1y_2, \ldots, y_{n-2}y_{n-1})$ and 
$\sigma_{n-1}(y_{n-1}y_n)= -y_{n-1}y_n$. Noting this we 
then have by Corollary~\ref{cor:sigma_signs} the following:
\begin{claim}
\label{clm:union}
For $n\geq 2$ we have 
$G_n = G_{n-1} \cup G_{n-1} \sigma_{n-1} = G_{n-1} \cup \sigma_{n-1} G_{n-1}$
where $\sigma_{n-1}(y_n \ES_{n-1}) = -y_n \ES_{n-1}$.
\end{claim}
If $G_n$ is presented as in Lemma~\ref{lmm:neg}, then
$\sigma_{n-1} \in G_n$ fixes 
$\mathbb{Q}(x_1, \ldots, x_n, y_1y_2, \ldots, y_{n-2}y_{n-1})$ and 
$\sigma_{n-1}(y_{n-1}y_n) = - y_{n-1}y_n$. 
\begin{lemma}
\label{lmm:ECnECn-1}
For $n\in\nats$ let $G_n$ be presented as in Lemma~\ref{lmm:neg}.
Then
\[
\left( \EC_n -1\right) \left(\sigma_{n-1}\left(\EC_n \right)-1\right)
= \left(x_n - \EC_{n-1}\right)^2. 
\]
In particular, $\EC_n = 1$ implies $x_n = \EC_{n-1}$.
\end{lemma}
\begin{proof}
By Claim~\ref{clm:union} we have 
$\sigma_{n-1}(y_n \ES_{n-1}) - -y_n \ES_{n-1}$ and hence
\begin{eqnarray*}
(\EC_n - 1 ) (\sigma_{n-1}(\EC_n)-1)	
& = & 
(x_n \EC_{n-1} - y_n \ES_{n-1}-1)(\sigma_{n-1}(x_n \EC_{n-1} 
- y_n \ES_{n-1})-1)\\
& = & 
(x_n \EC_{n-1} - y_n \ES_{n-1}-1)(x_n \EC_{n-1} + y_n \ES_{n-1}-1)\\
& = & (x_n \EC_{n-1} -1)^2 - y_n^2 \ES_{n-1}^2\\
& = & (x_n \EC_{n-1} -1)^2 - (1-x_n^2) (1-\EC_{n-1}^2)\\
& = & (x_n - \EC_{n-1})^2.
\end{eqnarray*}
\end{proof}
{\sc Remark:}
\label{rem:split}
Note that (\ref{eqn:2pi}) implies directly that 
$\cos(\theta_1 + \cdots + \theta_h) = \cos(\theta_{h+1} + \cdots + \theta_n)$ 
whenever $h+k=n$, and hence the equation 
$\EC_h(x_1, \ldots, x_h) = \EC_k(x_{h+1}, \ldots, x_n)$
which itself implies $x_n = \EC_{n-1}$ by letting $h=n-1$ and $k=1$.
\begin{corollary}
\label{cor:P_n}
For $n \geq 2$ we have $C_n = P_n^2$ where
\[
P_n := \prod_{\sigma \in G_{n-1}} (x_n - \sigma(\EC_{n-1})).
\]
\end{corollary}
\begin{proof}
By Lemma~\ref{lmm:ECnECn-1} we obtain:
\begin{eqnarray*}
C_n	
& = & 
\prod_{\sigma \in G_n} (\sigma(\EC_n) - 1)\\
& = & 
\prod_{\sigma \in \sigma_n G_{n-1} \cup G_{n-1}} (\sigma(\EC_n) - 1)\\
& = & \prod_{\sigma \in G_{n-1}} (\sigma(\EC_n) - 1)(\sigma_n\sigma(\EC_n) - 1)\\
& = & \prod_{\sigma \in G_{n-1}} \sigma((\EC_n - 1)(\sigma_n(\EC_n) - 1))\\
& = & \prod_{\sigma \in G_{n-1}} \sigma((x_n-\EC_{n-1})^2)\\
& = & \prod_{\sigma \in G_{n-1}} (x_n-\sigma(\EC_{n-1}))^2\\
& = & P_n^2
\end{eqnarray*}
where $P_n = \prod_{\sigma \in G_{n-1}} (x_n-\sigma(\EC_{n-1}))$.
\end{proof}
By exactly the same token as Claim \ref{clm:union}, Lemma \ref{lmm:ECnECn-1}, 
and Corollary \ref{cor:P_n}, we obtain analogous results by reordering 
the variables $y_1, \ldots, y_n$ in the reverse order: 
$y_n, y_{n-1}, \ldots, y_1$. Namely, if $\sigma_i \in G_n$ is the 
field automorphism of 
$\mathbb{Q}(x_1, \ldots, x_n, y_1y_2, y_2y_3, \ldots, y_{n-1}y_n)$ with 
$\sigma_i(y_i y_{i+1}) = -y_i y_{i+1}$ fixing 
$\mathbb{Q}(x_1, \ldots, x_n)$ and each $y_j y_{j+1}$ 
for $j \neq i$ (as in Lemma \ref{lmm:neg}) then 
we have the following:
\begin{claim}
\label{clm:union'}
If $n\geq 2$ then
$G_n = G'_{n-1} \cup \sigma_1G'_{n-1} = G'_{n-1} \cup G'_{n-1}\sigma_1$
where $G'_{n-1} = \left\langle \sigma_2, \ldots, \sigma_{n-1}\right\rangle$,
a subgroup of $G_n =  \left\langle \sigma_1, \ldots, \sigma_{n-1}\right\rangle$ and 
$\sigma_1(y_1 \ES_{n-1}(x_2, \ldots, x_n)) = -y_1 \ES_{n-1}(x_2, \ldots, x_n)$.
\end{claim}
\begin{proof}
By Lemma~\ref{lmm:sum}, each of the $2^{n-2}$ terms of 
$\ES_{n-1}(x_2, \ldots, x_n) =$ $\ES_{n-1}(x_2, \ldots, x_n, y_2, \ldots, y_n)$ 
(by substituting $y_i = \sqrt{1-x_i^2}$ for each $i=2, \ldots, n$) has 
positive odd factors of $y_i$ for $i \geq 2$. Hence the claim follows 
by Corollary \ref{cor:sigma_signs}.
\end{proof}
Similarly to Lemma~\ref{lmm:ECnECn-1} we now have
the following.
\begin{lemma}
\label{lmm:EC2-n}
If $\sigma_1 \in G_n$ is as above then 
\[
(\EC_n -1)(\sigma_1(\EC_n -1)) = (x_1- \EC_{n-1}(x_2, \ldots, x_n))^2.
\]
\end{lemma}
\begin{proof}
By Claim \ref{clm:union'} we obtain
\begin{eqnarray*}
(\EC_n -1)(\sigma_1(\EC_n -1))	
& = & 
(x_1 \EC_{n-1}(x_2, \ldots, x_n) - y_1 \ES_{n-1}(x_2, \ldots, x_n) - 1)\\
& & 
\cdot(\sigma_1(x_1 \EC_{n-1}(x_2, \ldots, x_n) - 
y_1 \ES_{n-1}(x_2, \ldots, x_n)) - 1)\\
& = & 
(x_1 \EC_{n-1}(\widehat{x}_1) - y_1 \ES_{n-1}(\widehat{x}_1) - 1)\\
& &
\cdot(x_1 \EC_{n-1}(\widehat{x}_1) + y_1 \ES_{n-1}(\widehat{x}_1) - 1)\\
& = & 
(x_1 \EC_{n-1}(\widehat{x}_1)-1)^2 - y_1^2 \ES_{n-1}(\widehat{x}_1)^2\\
& = & 
(x_1 \EC_{n-1}(\widehat{x}_1)-1)^2 - (1-x_1^2) (1-\EC_{n-1}(\widehat{x}_1)^2\\
& = & (x_1 - \EC_{n-1}(\widehat{x}_1))^2,
\end{eqnarray*}
where $(\widehat{x_1})= (x_2, \ldots, x_n)$ as above.
\end{proof}
\begin{corollary}
\label{cor:Cn=hat}
For $n \geq 3$ we have
\[
C_n=\prod_{\sigma \in G'_{n-1}}(x_1 - \sigma(\EC_{n-1}(\widehat{x}_1)))^2.
\]
\end{corollary}
\begin{proof}
By Lemma \ref{lmm:EC2-n} we obtain as in the proof of Corollary \ref{cor:P_n}
\begin{eqnarray*}
C_n	& = & \prod_{\sigma \in G_n} (\sigma(\EC_n) -1)\\
		& = & \prod_{\sigma \in G'_{n-1} \cup \sigma_1 G'_{n-1} } 
(\sigma(\EC_n) -1)\\
		& = & \prod_{\sigma \in G_{n-1}} (\sigma(\EC_n) -1)
(\sigma\sigma_1(\EC_n) -1)\\
		& = & \prod_{\sigma \in G_{n-1}} \sigma\left(((\EC_n) -1)
(\sigma_1(\EC_n) -1)\right)\\
		& = & \prod_{\sigma \in G_{n-1}} 
\sigma((x_1 -\EC_{n-1}(\widehat{x}_1))^2)\\
		& = & \prod_{\sigma \in G'_{n-1}} 
(x_1 - \sigma(\EC_{n-1}(\widehat{x}_1)))^2.
\end{eqnarray*}
\end{proof}
{\sc Remark:}
For $n=1$ we have $P_1 = C_1 = x_1 -1$. For $n=2$ we have (as defined in 
Corollary~\ref{cor:P_n}) $P_2= x_2 - x_1$. However, this is a matter of taste, 
since we could have set $P_2 = x_1 - x_2$. The case $n=2$ is the only one 
where $C_2(x_1, x_2)$ is symmetric while $P_2$ is not.

By Corollary \ref{cor:Cn=hat} we obtain $C_n = Q_n^2$ where 
\[
Q_n = \prod_{\sigma \in G_{n-1}} (x_1 - \sigma(\EC_{n-1}(\widehat{x}_1))).
\]  
Since $P_n^2 = C_n = Q_n^2$, then as elements in a polynomial ring over a 
field, an integral domain, we get 
$0 = P_n^2 - Q_n^2 = (P_n-Q_n)(P_n+Q_n)$ and hence for each 
$n \geq 2$ we have $Q_n = P_n$ or $Q_n = -P_n$.

For $n=2$ we obtain $P_2 = x_2 -x_1$ and $Q_2 = x_1 - x_2$ so $Q_2 = -P_2$.

For $n \geq 3$ we first note that by evaluating $\EC_{n-1}(\widehat{x}_n)$ and 
$\EC_{n-1}(\widehat{x}_1)$ at $x_2= \cdots = x_{n-1} = 1$ yields
$\EC_{n-1}(\widehat{x}_n) |_{x_2= \cdots = x_{n-1} = 1} = x_1$ and
$\EC_{n-1}(\widehat{x}_1) |_{x_2= \cdots = x_{n-1} = 1} = x_n$
and hence we obtain
\begin{eqnarray*}
P_n(x_1, 1, \ldots, 1, x_n)	
& = & \prod_{\sigma \in G_{n-1}} (x_n - x_1) = (x_n - x_1)^{2^{n-2}}\\
Q_n(x_1, 1, \ldots, 1, x_n)	
& = & \prod_{\sigma \in G_{n-1}} (x_1 - x_n) = (x_1 - x_n)^{2^{n-2}}.
\end{eqnarray*}
As $n \geq 3$, we have $2^{n-2}$ is even and so 
$(x_n - x_1)^{2^{n-2}} = (x_1 - x_n)^{2^{n-2}}$ and hence
$P_n(x_1, 1, \ldots, 1, x_n) = Q_n(x_1, 1, \ldots, 1, x_n)$.  
Therefore we obtain the following:
\begin{corollary}
\label{cor:P=Q}
For $n\geq 3$ we have $Q_n = P_n$ and hence 
\[
P_n = \prod_{\sigma \in G_{n-1}} (x_1 - \sigma(\EC_{n-1}(\widehat{x}_1))).
\]
\end{corollary}
We now want to show that for $n\geq 3$ the polynomial $P_n$ is symmetric.
Let $n \geq 3$. If $\pi \in S_n$ is a permutation on $\{1, \ldots, n\}$ 
then $\pi$ acts naturally on $(x_1, \ldots, x_n)$ by 
$\pi(x_1, \ldots, x_n)= (x_{\pi(1)}, \ldots, x_{\pi(n)})$.  
By definition of $P_n$ in Corollary~\ref{cor:P_n} we have 
\[
(P_n \circ \pi)(x_1, \ldots, x_n) = P_n(x_{\pi(1)}, \ldots, x_{\pi(n)}) 
= P_n(x_1, \ldots, x_n)
\]
or $P_n \circ \pi =P_n \pi= P_n$ for all 
$\pi \in S_n$ with $\pi(n)=n$. Likewise by Corollary \ref{cor:P=Q} we have 
$P_n \pi = P_n$ for all $\pi \in S_n$ with $\pi(1)= 1$.

Let $\tau \in S_n$ be an arbitrary transposition $\tau=(i,j)$.  If 
$\{i,j\}  \subseteq \{1, \ldots, n-1\}$ or $\{i,j\}  
\subseteq \{2, \ldots, n\}$ then by the above, $P_n   \tau = P_n$.  
Otherwise if $\tau = (1,n)$ then since $n\geq 3$ there is an 
$l \in \{2, \ldots, n-1\}$ such that we can write 
$\tau = (1,n) = (1,l)(l,n)(1,l)$ 
where $\{1, l\}  \subseteq \{2, \ldots, n\}$.  From the above, we 
therefore have
\[
P_n\tau=P_n(1,n)=P_n(1,l)(l,n)(1,l)=P_n (l,n)(1,l)=P_n(1,l)=P_n.
\]
Since each permutation $\pi \in S_n$ is a composition of transpositions 
then we have $P_n \pi = P_n$ for each $\pi \in S_n$.
\begin{theorem}
\label{thm:P_n=sym}
For $n\geq 3$ the polynomial $P_n = P_n (x_1, \ldots, x_n)$ is symmetric.
\end{theorem}
\begin{corollary}
\label{cor:P_n-x_i}
For $n \geq 3$ and any $i \in \{1, \ldots, n\}$ we have
\[
P_n=\prod_{\sigma \in G_{n-1}} \sigma(x_i - \EC_{n-1}(\widehat{x_i})).
\]
In particular, as a polynomial in $x_i$, then $P_n$ is monic of degree 
$2^{n-2}$ in each $x_i$.
\end{corollary}
By Corollary~\ref{cor:P_n-x_i} and definition of $C_{n-1}$ 
we obtain by letting $x_i = 1$ the following:
\begin{observation}
\label{obs:xi=1}
For $n\geq 3$ then
\[
P_n(x_1, \ldots, x_{i-1}, 1, x_{i+1}, \ldots, x_n)	
= \prod_{\sigma \in G_{n-1}} (1 - \sigma(\EC_{n-1}(\widehat{x_i}))) 
= C_{n-1}(\widehat{x_i})
= P_{n-1}(\widehat{x_i})^2.
\]
\end{observation}
Other more general equations and formulae hold as well.
Let $n\in\nats$ and $n_1 + \cdots + n_k = n$.  If $\sum_{i=1}^n \theta_i = 2\pi$ 
and $x_i=\cos \theta_i$ for each $i \in \{1, \ldots, n\}$, then
for each $l \in \{1, \ldots, k\}$ let
$\phi_l = \theta_{n_1+\cdots+n_{l-1}+1} + \cdots + \theta_{n_1+\cdots+n_{l}}$.
Then $\sum_{l=1}^k \phi_l = 2\pi$
and hence if $t_l= \cos(\phi_l)$ then by Corollary~\ref{cor:P_n-x_i}
we get 
\[
0 = P_k(t_1, \ldots, t_l) = P_k(\EC_{n_1}, \ldots, \EC_{n_k}),
\]
where for each $l \in \{1, \ldots, k\}$ we have
$\EC_{n_l} = \EC_{n_l}(x_{n_1 + \cdots + n_{l-1}+1}, \ldots, x_{n_1 + \cdots + n_l}).$
In particular for $k=n-1$ and $n_1 = \cdots = n_{n-2} = 1$ and $n_{n-1}=2$, 
we have $P_{n-1}(x_1, \ldots, x_{n-2}, \EC_2(x_{n-1}, x_n))=0$,
something we can use to compute $P_n$ recursively. 
Let $\overline{\EC_2}(x_j, x_{j+1}) = x_jx_{j+1} + y_jy_{j+1}$ be the 
conjugate of $\EC_2(x_j, x_{j+1})$.  Recall that by 
Claim~\ref{clm:union} we have for $n-1$ that 
\[
G_{n-1} = G_{n-2} \cup \sigma_{n-1} G_{n-2} = G_{n-2} \cup G_{n-2}\sigma_{n-1}
\]
and $\sigma_{n-2}(y_{n-1}\ES_{n-2}) = -y_{n-1}\ES_{n-2}$.
\begin{lemma}
\label{lmm:pre-rec}
For $n\geq 3$ we have
\[
(x_n - \EC_{n-1})(x_n - \sigma_{n-2}(\EC_{n-1}))	
= x_{n-1}^2 + x_n^2 -1 - 2x_{n-1} x_n \EC_{n-2}^2 + \EC_{n-2}^2.
\]
\end{lemma}
\begin{proof}
Since $\EC_{n-1} = x_{n-1} \EC_{n-2} - y_{n-1} \ES_{n-2}$, we obtain by above
\begin{eqnarray*}
(x_n - \EC_{n-1})(x_n - \sigma_{n-2}(\EC_{n-1}))	
& = & (x_n - x_{n-1} \EC_{n-2} + y_{n-1} \ES_{n-2})\\
&   & \cdot(x_n - x_{n-1} \EC_{n-2} - y_{n-1} \ES_{n-2})\\
& = & (x_n - x_{n-1} \EC_{n-2})^2 - y_{n-1}^2 \ES_{n-2}^2\\
& = & (x_n - x_{n-1} \EC_{n-2})^2 - (1-x_{n-1}^2)(1- \EC_{n-2}^2)\\
& = & x_{n-1}^2 + x_{n}^2 - 1- 2x_{n-1}x_n \EC_{n-2}^2 + \EC_{n-2}^2.
\end{eqnarray*}
\end{proof}
By direct computation and the definition of $P_{n-1}$, since 
$\EC_2(x_i, x_{i+1}) = x_i x_{i+1} - y_i y_{i+1}$, we get
\begin{eqnarray*}
P_{n-1}(x_1, &\ldots,& x_{n-2}, \EC_2(x_{n-1}x_n))P_{n-1}(x_1, 
\ldots, x_{n-2}, \overline{\EC_2}(x_{n-1}x_n))\\
& = & \prod_{\sigma \in G_{n-2}} (\EC_2(x_{n-1}, x_n) - 
\sigma(\EC_{n-2}))\prod_{\sigma \in G_{n-2}} 
(\overline{\EC_2}(x_{n-1}, x_n) - \sigma(\EC_{n-2}))\\
& = & \prod_{\sigma \in G_{n-2}} (x_{n-1}x_n - y_{n-1}y_n - 
\sigma(\EC_{n-2}))\prod_{\sigma \in G_{n-2}} 
(x_{n-1}x_n + y_{n-1}y_n - \sigma(\EC_{n-2}))\\
& = & \prod_{\sigma \in G_{n-2}} \left((x_{n-1}x_n - 
\sigma(\EC_{n-2}))^2 - y_{n-1}^2y_n^2\right)\\
& = & \prod_{\sigma \in G_{n-2}} \left((x_{n-1}x_n - 
\sigma(\EC_{n-2}))^2 - (1-x_{n-1}^2)(1-x_{n}^2)\right)\\
& = & \prod_{\sigma \in G_{n-2}} \left(x_{n-1}^2 + x_{n}^2 - 1- 2x_{n-1}x_n 
\sigma(\EC_{n-2}^2) + \sigma(\EC_{n-2}^2)\right)\\
& = & \prod_{\sigma \in G_{n-2}} \sigma\left(x_{n-1}^2 + x_{n}^2 - 1- 2x_{n-1}x_n 
\EC_{n-2}^2 + \EC_{n-2}^2\right).
\end{eqnarray*}
From this we can prove the following:
\begin{theorem}
\label{thm:recursion1}
The polynomials $P_n$ are completely determined by the following recursion:
$P_1= x_1 - 1$, $P_2= x_2 - x_1$ and for $n\geq 3$
\[
P_n=P_{n-1}(x_1, \ldots, x_{n-2}, \EC_2(x_{n-1}, x_n))P_{n-1}(x_1, 
\ldots, x_{n-2}, \overline{\EC_2}(x_{n-1}, x_n)).
\]
\end{theorem}
\begin{proof}
By Lemma \ref{lmm:pre-rec} and the preceding paragraph we get
\begin{eqnarray*}
P_n & = & \prod_{\sigma \in G_{n-1}} (x_n - \sigma(\EC_{n-1}))\\
& = & 
\prod_{\sigma \in G_{n-2} \cup \sigma_{n-1} G_{n-2}} (x_n - \sigma(\EC_{n-1}))\\
& = & \prod_{\sigma \in G_{n-2}} (x_n - \sigma(\EC_{n-1})) 
(x_n - \sigma\sigma_{n-1}(\EC_{n-1}))\\
& = & \prod_{\sigma \in G_{n-2}} \sigma\left((x_n - (\EC_{n-1})) 
(x_n - \sigma_{n-1}(\EC_{n-1}))\right)\\
& = & \prod_{\sigma \in G_{n-2}} 
\sigma\left(x_{n-1}^2 + x_n^2 -1 -2x_{n-1}x_n \EC_{n-2}^2 + \EC_{n-2}^2\right)\\
& = & P_{n-1}(x_1, \ldots, x_{n-2}, \EC_2(x_{n-1}, x_n))
\cdot P_{n-1}(x_1, \ldots, x_{n-2}, \overline{\EC_2}(x_{n-1}, x_n)).
\end{eqnarray*}
\end{proof}
{\sc Example:}
With the help of MAPLE~\cite{MAPLE} the first 5 polynomials $P_n$ can now be 
computed quickly and efficiently by the recursion in Theorem \ref{thm:recursion1}.
\begin{eqnarray*}
P_1 & = & x_1 - 1.\\ 
P_2 & = & x_2 - x_1.\\ 
P_3 & = & P_2(x_1,\EC_2(x_2, x_3))P_2(x_1,\overline{\EC_2}(x_2, x_3))\\
& = & (x_2x_3 - y_2y_3 - x_1)(x_2x_3 + y_2y_3 - x_1)\\
& = & x_1^2 + x_2^2 + x_3^2 - 2x_1x_2x_3 - 1.\\ 
P_4	
& = & P_3(x_1,x_2,\EC_2(x_3,x_4))P_3(x_1,x_2,\overline{\EC_2}(x_3, x_4))\\
& = & (x_1^2 + x_2^2 + (x_3x_4 - y_3y_4)^2 - 2x_1x_2(x_3x_4 - y_3y_4) - 1)\\
&   & \cdot(x_1^2 + x_2^2 + (x_3x_4 + y_3y_4)^2 - 2x_1x_2(x_3x_4 + y_3y_4) - 1)\\
& = & x_1^4 + x_2^4 + x_3^4 + x_4^4 - 
2(x_1^2x_2^2+x_2^2x_3^2 + x_3^2x_4^2 + x_1^2x_4^2 + x_1^2x_3^2 + x_2^2x_4^2) \\
&   & + 4(x_1^2x_2^2x_3^2 + x_2^2x_3^2x_4^2 + x_1^2x_3^2x_4^2 + x_1^2x_2^2x_4^2) \\
&   & + 4x_1x_2x_3x_4(2-x_1^2-x_2^2-x_3^2-x_4^2).\\
P_5
& = & \mbox{a display of terms on two letter size pages, see Appendix~\ref{appx:P5}.}
\end{eqnarray*}

The recursion given in Theorem~\ref{thm:recursion1}, although 
fundamental for computation, is a special case of a more general 
recursion that $P_n$ satisfies:
\begin{claim}
\label{clm:recursion2}
Let $n,k\geq 2$ and $n_1 + \cdots + n_k = n$.
By the right interpretation of $\sigma_i$
for each $i\in[k]$ (and with some abuse of notation)  
then $P_n = P_n(x_1, \ldots, x_n)$ satisfies the following
general recursion
\begin{eqnarray*}
P_n(x_1, \ldots, x_n)  =  \prod_{\stackrel{\sigma_i \in G_{n_i-1}}{ i\in[k]}} &&  
P_k\left(\sigma_1(\EC_{n_1}(x_1, \ldots, x_{n_1})), \sigma_2(\EC_{n_2}(x_{n_1+1}, 
\ldots,x_{n_1+n_2})),\right. \\
&& \left.\ldots, \sigma_k(\EC_{n_k}(x_{n-n_k+1}, \ldots, x_n))\right). 
\end{eqnarray*}
\end{claim}
As this more general recursion of Claim~\ref{clm:recursion2}
will not be used to obtain our main result Theorem~\ref{thm:irreducible}
here below, its proof in detail will be omitted. However, 
this can be proved using induction in stages using
Theorem~\ref{thm:recursion1} as a stepping stone.

{\sc Example:}
We demonstrate how Claim~\ref{clm:recursion2} works by using it
to compute $P_5$, since $n=5$ is the smallest nontrivial example
(with $k\geq 3$) that can be generated using a recurrence from 
Claim~\ref{clm:recursion2} that is not an example of the special 
recurrence from  Theorem~\ref{thm:recursion1}:
\begin{eqnarray*}
P_5(x_1, \ldots, x_5)	
& = &	
\prod_{\stackrel[\sigma_3'\in G_{3}=\langle\sigma_3\rangle]{\sigma_1'\in G_{1}=\langle\sigma_1\rangle}{\sigma_2'\in G_{2}= \{e\}}}
P_3(\sigma_1'(\EC_2(x_1,x_2)),\sigma_2'(\EC_1(x_3)),\sigma_3'(\EC_2(x_4,x_5)))\\
& = & 
\prod_{\stackrel[\sigma_3'\in G_{3}=\langle\sigma_3\rangle]{\sigma_1'\in G_{1}=\langle\sigma_1 
\rangle}{\sigma_2' \in G_{2}= \{e\}}}
P_3(\sigma_1'(x_1x_2 - y_1y_2), \sigma_2'(x_3), \sigma_3'(x_4x_5 - y_4y_5))\\
& = & P_3(x_1x_2 - y_1y_2, x_3, x_4x_5 - y_4y_5) \cdot 
P_3(x_1x_2 + y_1y_2, x_3, x_4x_5 - y_4y_5)\\
&   & \cdot P_3(x_1x_2 - y_1y_2, x_3, x_4x_5 + y_4y_5) \cdot 
P_3(x_1x_2 + y_1y_2, x_3, x_4x_5 + y_4y_5)\\
& = & ((x_1x_2 - y_1y_2)^2 + x_3^2 + (x_4x_5 - y_4y_5)^2 - 
2(x_1x_2 - y_1y_2)x_3(x_4x_5 - y_4y_5) - 1)\\
&   & \cdot((x_1x_2 + y_1y_2)^2 + x_3^2 + (x_4x_5 - y_4y_5)^2 - 
2(x_1x_2 + y_1y_2)x_3(x_4x_5 - y_4y_5) - 1)\\
&   & \cdot((x_1x_2 - y_1y_2)^2 + x_3^2 + (x_4x_5 + y_4y_5)^2 - 
2(x_1x_2 - y_1y_2)x_3(x_4x_5 + y_4y_5) - 1)\\
&   & \cdot((x_1x_2 + y_1y_2)^2 + x_3^2 + (x_4x_5 + y_4y_5)^2 - 
2(x_1x_2 + y_1y_2)x_3(x_4x_5 + y_4y_5) - 1).
\end{eqnarray*}
Expanded, this last product yields the same expression for $P_5$ 
as given in Appendix~\ref{appx:P5}.

\vspace{3 mm}

Our final goal in this section, and our main result of the paper, is to 
prove the irreducibility of $P_n$.  
To illuminate our approach we state and prove the following simplest case, 
that $P_3 = P_3(x_1, x_2, x_3)$ is irreducible.

Suppose $P_3 = fg$ with $f,g \in \mathbb{Q}[x_1, x_2, x_3]$.  Since $P_3$ 
is monic in $x_3$, both $f$ and $g$ contain the variable $x_3$, and hence 
both $f$ and $g$ are of degree 1 in $x_3$ (unless $f$ or $g=P_3$.)  Since 
$P_3$ factors in $\mathbb{Q}(x_1, x_2, y_1y_2)[x_3]$ as 
$P_3=(x_3-x_1x_2-y_1y_2)(x_3-x_1x_2+y_1y_2)$ by definition of $P_3$, 
then since $\mathbb{Q}(x_1,x_2,y_1y_2)[x_3]$ is a UFD we must have 
\[
\{f,g\} = \{x_3-x_1x_2-y_1y_2, x_3-x_1x_2+y_1y_2\}
\]
which contradicts the assumption that 
$f,g \in \mathbb{Q}[x_1, x_2, x_3]$.  Hence we have 
the following observation:
\begin{observation}
The polynomial $P_3(x_1, x_2, x_3)$ is irreducible over $\mathbb{Q}$.
\end{observation}
Note that the same argument holds if $\mathbb{Q}$ is replaced
with the complex field $\mathbb{C}$ in the above.

We now use this same approach to prove the following:
\begin{theorem}
\label{thm:irreducible}
For each $n \geq 3$ the polynomial $P_n(x_1, \ldots, x_n)$ is 
irreducible over $\mathbb{Q}$.
\end{theorem}
We will prove Theorem~\ref{thm:irreducible} by induction on $n$, 
assuming that $P_{n-1}$ is irreducible over $\mathbb{Q}$. But
before we can delve into that, we need to prove the following:
\begin{lemma}
\label{lmm:h*}
Let $n \geq 3$. If $P_{n-1}$ is irreducible over $\mathbb{Q}$ then
$P_{n-1}(\EC_2(x_1, x_2), x_3, \ldots, x_n) 
=P_{n-1}(x_1x_2-y_1y_2, x_3, \ldots, x_n)$ and 
$P_{n-1}(\overline{\EC_2}(x_1, x_2), x_3, \ldots, x_n)	
=P_{n-1}(x_1x_2+y_1y_2, x_3, \ldots, x_n)$
are irreducible in $\mathbb{Q}(x_1, x_2, y_1y_2)[x_3, \ldots, x_n]$.
\end{lemma}
\begin{proof}
Let $P_{n-1}^*:=P_{n-1}(x_1x_2-y_1y_2, x_3, \ldots, x_n)$ and assume it
factors as $P_{n-1}^*= h^*k^*$ in the ring $\mathbb{Q}(x_1, x_2, y_1y_2)[x_3, \ldots, x_n]$, 
where both $h^*$ and $k^*$ involve $x_n$. Since
$P_{n-1}=\prod_{\sigma \in G_{n-2}} (x_n - \sigma(\EC_{n-2}))$,
we see that 
\[
P_{n-1}^*=\prod_{\sigma \in G_{n-2}} (x_n - 
\sigma(\EC_{n-2}(x_1x_2-y_1y_2, x_3, \ldots, x_n))),
\]
and hence both $h^*$ and $k^*$ must be products of these linear factors.
In particular, we can evaluate $P_{n-1}^* = h^*k^*$ at $x_1 = 1$ and obtain
\[
P_{n-1}(x_2, \ldots, x_n)=\left(P_{n-1}^* \right) |_{x_1=1} = 
\left(h^* |_{x_1-1} \right)\left(k^* |_{x_1-1} \right)= hk
\]
in $\mathbb{Q}(x_2)[x_3, \ldots, x_n]$, which is a UFD.
By assumption $P_{n-1}(x_2, \ldots, x_n)$ is irreducible
in  the ring $\mathbb{Q}[x_2, \ldots, x_n] = \mathbb{Q}[x_2][x_3, \ldots, x_n]$
and hence also in $\mathbb{Q}(x_2)[x_3, \ldots, x_n]$ (as a monic polynomial
in $x_n$). Therefore
either $h$ or $k$ equals 
$P_{n-1}(x_2, \ldots, x_n)$, which contradicts the fact that both $h^*$ and 
$k^*$ involve $x_n$.  Hence $P_{n-1}^*$ is irreducible. In the same way we 
obtain that $P_{n-1}(x_1x_2+y_1y_2, x_3, \ldots, x_n)$ is irreducible.
\end{proof}
\begin{proof}[Theorem~\ref{thm:irreducible}]
Let $n\geq 3$ and assume that $P_{n-1}$ is irreducible over $\mathbb{Q}$.
Assume $P_n= fg$ with $f, g \in\mathbb{Q}[x_1, \ldots, x_n]$. We may assume 
$f$ is irreducible. Let $\phi_i: \mathbb{Q}[x_1, \ldots, x_n]
\longrightarrow\mathbb{Q}[\widehat{x_i}]$ be the evaluation at
$x_i=1$, that is $\phi_i(F)=F(x_1, \ldots, x_{i-1}, 1, x_{i+1}, \ldots, x_n)$.
Since $\phi_i$ is a $\mathbb{Q}$-algebra homomorphism for each 
$i\in[n]$ we have for $i=1$ that
\[
\phi_1(P_n)= \phi_1(fg)	= \phi_1(f) \phi_1(g)\in \mathbb{Q}[x_2, \ldots, x_n].
\]
But $\phi_1(P_n) = P_{n-1}(x_2, \ldots, x_n)^2 \in 
\mathbb{Q}[x_2, \ldots, x_n]$, which is a UFD.  By the inductive 
hypothesis, $P_{n-1}$ is irreducible in $\mathbb{Q}[x_2, \ldots, x_n]$.  
Therefore, $\phi_1(f) = P_{n-1} = \phi_1(g)$ (unless $f= P_n$, in which
case we are done since $f$ is irreducible).

Viewing $f,g \in \mathbb{Q}[x_1, \ldots, x_{n-1}][x_n]$, then since 
$P_n$ and $P_{n-1}$ are monic in every variable $x_i$ 
(and hence also in $x_n$,) we have
\[
\deg_{x_n}(f)	= \deg_{x_n}(g) = \frac{\deg_{x_n}(P_n)}{2} = 2^{n-3}.
\]
By symmetry of $P_n$ for $n \geq 3$, from Theorem \ref{thm:P_n=sym} and 
Theorem \ref{thm:recursion1} we have
\[
P_n= P_{n-1}(\EC_2(x_1, x_2), x_3, \ldots, x_n) 
P_{n-1}(\overline{\EC_2}(x_1, x_2), x_3, \ldots, x_n)
\]
in $\mathbb{Q}(x_1, x_2, y_1y_2)[x_3, \ldots, x_n]$, which is a UFD.
Since by assumption $P_n= fg$ where 
$f \in \mathbb{Q}[x_1, \ldots, x_n]$ is irreducible and 
$f |_{x_1=1} = \phi_1(f) = P_{n-1}(x_2, \ldots, x_n)$, which
by assumption is irreducible. That $f$ is also irreducible in 
$\mathbb{Q}(x_1, x_2, y_1y_2)[x_3, \ldots, x_n]$ can now be
seen in the same way as in the proof of Lemma~\ref{lmm:h*}:
namely, by evaluating at $x_1=1$ and obtain a factorization of 
$P_{n-1}(x_2, \ldots, x_n)$.
 
So we have 
\[
P_n = fg = P_{n-1}(\EC_2, x_3, \ldots, x_n) 
P_{n-1}(\overline{\EC_2}, x_3, \ldots, x_n)
\] 
in $\mathbb{Q}(x_1, x_2, y_1y_2)[x_3, \ldots, x_n]$, which is a UFD. 
Therefore we have 
\[
f \in \{ P_{n-1}(\EC_2, x_3, \ldots, x_n), 
P_{n-1}(\overline{\EC_2}, x_3, \ldots, x_n)\},
\]
By repeated application of Observation~\ref{obs:xi=1} we
obtain
\[
P_{n-1}(\EC_2, 1, \ldots, 1) = P_1(\EC_2)^{2^{n-2}} = (\EC_2 - 1)^{2^{n-2}}
\]
which is not contained in $\mathbb{Q}[x_1, \ldots, x_n]$. Similarly
$P_{n-1}(\overline{\EC_2}, x_3, \ldots, x_n)\not\in\mathbb{Q}[x_1, \ldots, x_n]$
and hence we have a contradiction, since $f \in \mathbb{Q}[x_1, \ldots, x_n]$.
\end{proof}
{\sc Remark:} Replacing $\mathbb{Q}$ with $\mathbb{C}$ in the previous proofs will yield the same result.

As a corollary we obtain the following, which in fact equivalent to 
Theorem \ref{thm:irreducible}:
\begin{corollary}
For $n\in\nats$ we have
$[\mathbb{Q}(x_1, \ldots, x_n, \EC_n) : \mathbb{Q}(x_1, \ldots, x_n)]=2^{n-1}$.

In fact, for any $m \leq n$ we have
$[\mathbb{Q}(x_1, \ldots, x_n, \EC_m) : \mathbb{Q}(x_1, \ldots, x_n)]=2^{m-1}$.
\end{corollary}
we conclude this section with a summarizing result:
\begin{corollary}
For $1 \leq m \leq n$ we have
\begin{itemize}
\item $\mathbb{Q}(x_1, \ldots, x_n, \EC_m)= \mathbb{Q}(x_1, 
\ldots, x_n, y_1y_2, \ldots, y_{m-1}y_m)$
\item \ \ \ $\Gal(\mathbb{Q}(x_1, \ldots, x_n, \EC_m)/\mathbb{Q}(x_1, 
\ldots, x_n))$\\
$= \Gal(\mathbb{Q}(x_1, \ldots, x_n, y_1y_2, 
\ldots, y_{m-1}y_m)/\mathbb{Q}(x_1, \ldots, x_n))$\\
$ \cong  \mathbb{Z}_2^{m-1}$.
\item $P_{m+1}(x_1, \ldots, x_m, X) \in \mathbb{Q}(x_1, \ldots, x_m)[X]$ 
is the minimal polynomial of\\
$\EC_m = \EC_m(x_1, \ldots, x_m)$ over $\mathbb{Q}(x_1, \ldots, x_m)$.
\end{itemize}
\end{corollary}

\section{Rational solutions for $n=3$ and Descartes' circle theorem}
\label{sec:n=3}

Here we deal with the special case of $n=3$, and we characterize all rational 
solutions for flowers with three petals, thereby obtaining all rational
radii of four mutually tangent Soddy circles. 
We then compare our parametrization to an existing parametrization of the 
curvatures of four mutually tangent circles and show how our equation-free 
parameterization is an improvement on the existing one.

In general, to find all integer-radii coins forming an $n$-petal flower in the Euclidean plane, 
it is equivalent by scaling, to find all rational radii
coins where the center coin is assumed to have radius
one. Note that if the lengths of the sides of a triangle are 
rational, then the cosines of all its angles will be rational.
The converse is not necessarily 
true, however. We will solve $P_3=0$ over the rationals and
use that to find rational radii that create a 3-petal flower,
that is, rational Soddy circles.
For a necessary first step, we will determine what the cosines must be.

\subsection{Rational solutions of $P_3=0$ and rational Soddy circles}

First, we have the irreducible polynomial 
$P_3 = x_1^2+x_2^2+x_3^2 -2x_1x_2x_3-1$.  We can solve for any one of the 
variables, say $x_3$, by definition of $P_3$ and Lemma~\ref{lmm:ECnECn-1}: 
\begin{equation}
\label{eqn:x3}
x_3 = \EC_2(x_1, x_2)  =  x_1 x_2 - \sqrt{(1-x_1^2)(1-x_2^2)}.
\end{equation}
For rational $x_1$ and $x_2$ it is clear that $x_3$ will be rational 
if and only if the term under the radical is the square of a rational number.
For $i=1,2$ let $x_i = \frac{p_i}{q_i}$ for with $p_i, q_i \in \mathbb{Z}$.  
By (\ref{eqn:x3}) we then obtain 
\[
x_3 = x_1 x_2 - \frac{1}{q_1 q_2}\sqrt{(q_1^2-p_1^2)(q_2^2-p_2^2)}.
\]
Here $(q_1^2-p_1^2)(q_2^2-p_2^2)$ is a square if and only if
$q_i^2-p_i^2 = s_i^2 \beta$ for $i=1,2$ where $\beta$ is 
square-free integer. Here we need the following
result in elementary number theory:
\begin{theorem}
\label{thm:pyth_ext}
Let $\beta$ be a square-free integer.  The integers $x,y,z$ form a 
primitive solution to the Diophantine equation $x^2 + \beta y^2 = z^2$ 
if and only if there are positive integers $m$ and $n$ and a factorization 
$\beta= b c$ where $b m^2$ and $c n^2$ are relatively prime such that
$x = \frac{b m^2-c n^2}{2}$, $y =  mn$, $z = \frac{b m^2+c n^2}{2}$,
where both $m$ and $n$ are odd or both are even, or
$x=b m^2-c n^2$, $y=2mn$, $z =b m^2+c n^2$ otherwise.
\end{theorem}
For a proof of Theorem~\ref{thm:pyth_ext}, see Appendix~\ref{appx:beta}.

Since $x_i = \frac{p_i}{q_i}$ for $i=1,2$ we have by 
Theorem \ref{thm:pyth_ext} that
\begin{equation}
\label{eqn:x1x2}
x_1 = \frac{b_1 m_1^2-c_1 n_1^2}{b_1 m_1^2+c_1 n_1^2}, \ \  
x_2 = \frac{b_2 m_2^2-c_2 n_2^2}{b_2 m_2^2+c_2 n_2^2},
\end{equation}
where $\beta = b_1 c_1 = b_2 c_2$ are two (not necessarily distinct) factorizations of the square-free 
integer $\beta$, and where $m_i, n_i$ can be chosen from the nonnegative 
integers.  

Suppose we have a 3-petal flower whose internal angles are 
$\theta_1, \theta_2, \theta_3$ and their cosines are $x_1,x_2,x_3$
respectively. By scaling, we assume the radius of the center coin
to be one and the other three outer radii $r_1, r_2$ and $r_3$.
By the law of cosines, we obtain
\[
x_1 = \frac{r_1 + r_2 -r_1r_2 +1}{r_1 + r_2 + r_1 r_2 + 1}, \ \ 
x_2 = \frac{r_2 + r_3 -r_2r_3 +1}{r_2 + r_3 + r_2 r_3 + 1}, \ \ 
x_3 = \frac{r_3 + r_1 -r_3r_1 +1}{r_3 + r_1 + r_3 r_1 + 1}.
\]
Rewriting each equation for $x_i$ as a polynomial equation in 
terms of $r_i$ and $r_{i+1}$ ( where $4 \equiv 1$ modulo 3) and then
factoring in terms of $r_i$ and $r_{i+1}$ we obtain
\begin{eqnarray*}
\left(r_1 + \frac{x_1-1}{x_1+1} \right) 
\left(r_2 + \frac{x_1-1}{x_1+1} \right) & = & \frac{2(1-x_1)}{(x_1+1)^2},\\
\left(r_2 + \frac{x_2-1}{x_2+1} \right) 
\left(r_3 + \frac{x_2-1}{x_2+1} \right) & = & \frac{2(1-x_2)}{(x_2+1)^2},\\
\left(r_3 + \frac{x_3-1}{x_3+1} \right) 
\left(r_1 + \frac{x_3-1}{x_3+1} \right) & = & \frac{2(1-x_3)}{(x_3+1)^2}.
\end{eqnarray*}
Now we can solve the first and third equations for $r_2$ and $r_3$ 
respectively in terms of $r_1, x_1, x_3$.  Substituting these into the 
second equation, we can then solve that for $r_1$ in terms of $x_1, x_2, x_3$ 
obtaining
\begin{equation}
\label{eqn:r1}
r_1 = \frac{-1-x_1 x_3+x_3+x_1 \pm 
\sqrt{2(1-x_1)(1-x_2)(1-x_3)}}{2 x_2-x_1+x_1 x_3-1-x_3}.
\end{equation}
Putting $x_1$ and $x_2$ from (\ref{eqn:x1x2}) into 
(\ref{eqn:x3}) we obtain 
\[
x_3= x_1 x_2 - \sqrt{(1-x_1^2)(1-x_2^2)}= \frac{\left(b_1 m_1^2-c_1 n_1^2\right)
\left(b_2 m_2^2-c_2 n_2^2\right) -  4 m_1 m_2 n_1 n_2 \beta}{(b_1 m_1^2+
c_1 n_1^2)(b_2 m_2^2+c_2 n_2^2)}.
\]
Substituting this expressions for $x_3$ and those of $x_1$
and $x_2$ from (\ref{eqn:x1x2})
into (\ref{eqn:r1}), we get an expression for $r_1$ in terms of 
$b_1, b_2, c_1, c_2, m_1, m_2, n_1, n_2$:
\begin{eqnarray*}
r_1 
& = & 
\frac{n_1(b_2 c_1^2 m_2^2 n_1^3  +2\beta c_1  m_1 m_2 n_1^2 n_2 +
b_1 c_1 c_2 m_1^2 n_1  n_2^2 )}{b_1 c_1 c_2 m_1^2 n_1^2 n_2^2  
-b_2 c_1^2 m_2^2  n_1^4+c_1^2 c_2 n_1^4 n_2^2 -2 \beta c_1 m_1 m_2 n_1^3 n_2 
+b_1^2 c_2 m_1^4n_2^2 } \\ \\
& \pm & 
\frac{ n_1 n_2 (b_1 m_1^2+c_1 n_1^2) 
\sqrt{c_1 c_2 (b_1 c_2 m_1^2 n_2^2+2 \beta m_1 m_2  n_1n_2 +
b_2c_1 m_2^2 n_1^2  )}}{b_1c_1 c_2 m_1^2 n_1^2 n_2^2 -b_2 c_1^2 m_2^2 n_1^4+
c_1^2 c_2 n_1^4 n_2^2 -2 \beta c_1 m_1 m_2 n_1^3 n_2 +b_1^2 c_2 m_1^4 n_2^2 }.
\end{eqnarray*}
Using the fact that $\beta = b_1 c_1 = b_2 c_2$, the expression under
the square root can be reduced to $\beta (c_2 m_1 n_2 + c_1 m_2 n_1)^2$.
Thus this expression for $r_1$ will only yield a perfect square 
when $\beta = 1$. Therefore, $r_1$ is rational if and only if 
$q_i^2-p_i^2 = s_i^2$ for $i=1,2$, or in other words when 
$1-x_i^2$ is a perfect square for $i=1,2$. This means that both 
$\cos\theta_i$ and $\sin\theta_i$ are rational for $i=1,2,3$.
\begin{proposition}
\label{prp:cos-sin-rat}
The 3-petaled flower with the center coin of radius one 
can have the outer coins of rational radii $r_1, r_2, r_3$
if and only if the internal angles $\theta_1, \theta_2, \theta_3$
have both rational cosines and sines for $i=1,2,3$.
\end{proposition}
Proposition~\ref{prp:cos-sin-rat} shows a property
that is very special for the $n$-petal flower with
rational radii when $n=3$. We now can write a ``nice''
parametrization for the cosines $x_i$ and the radii $r_i$ 
in the case when $n=3$.  Let $m_1, n_1, m_2, n_2 \in \mathbb{N}$.  
Then
\begin{equation}
\label{eqn:xi-nice}
x_1=\frac{m_1^2-n_1^2}{m_1^2+ n_1^2}, \ \
x_2=\frac{m_2^2-n_2^2}{m_2^2+ n_2^2}, \ \  
x_3=\frac{\left(m_1^2-n_1^2\right)\left(m_2^2-n_2^2\right) -  
4 m_1 m_2 n_1 n_2}{(m_1^2+n_1^2)(m_2^2+n_2^2)}.
\end{equation} 
Putting these into (\ref{eqn:r1}) and the similar
equations for $r_2$ and $r_3$ we obtain the rational
forms for $r_1,r_2,r_3$ that cointain all rational radii
for the outer coins of a 3-petal flower with center coin
of radius one:
\begin{eqnarray*}
r_1 & = & \frac{n_1(m_1 n_2 + m_2 n_1)}{ - m_1^2n_2 - n_1^2 n_2 
\pm (m_1 n_1 n_2  +m_2 n_1^2 )}\\
r_2	& = & \frac{n_1 n_2}{ - n_1 n_2 \pm (m_2 n_1+  m_1n_2)}\\
r_3 & = & \frac{n_2(m_1n_2 +m_2n_1 )}{ - m_1 n_2^2  -m_2 n_1n_2  
\pm(n_1n_2^2 +m_2^2 n_1 )}.
\end{eqnarray*}
We will determine which range of the parameters will yield
meaningful solutions in what follows as well as
the signs of the terms in the denominator.
\begin{observation}
\label{obs:90180}
If $\theta_1, \theta_2, \theta_3$ are the internal angles 
of a 3-petaled flower, then $90^{\circ} < \theta_i < 180^{\circ}$ 
for each $i=1,2,3$ and these three inequalities 
are all sharp.
\end{observation}
\begin{proof}
As each $\theta_i$ is an angle in a triangle formed by 
the mutually touching three coins, we have that
$\theta_i < 180^{\circ}$. On the other hand, keeping
the radius of the center coin fixed (say, at $r=1$) and
letting $r_i = r_{i+1} \rightarrow\infty$, we see
that $\theta_i\rightarrow 180^{\circ}$ from below. We also see
from this scenario that the other two angles tend
to $90^{\circ}$ from below.

What remains to show is that $\theta_i > 90^{\circ}$ for each $i$.
It suffices to show this for $i=1$. By keeping the 
radii $r_1$ and $r_2$ fixed and letting $r_3 \rightarrow \infty$, the 
radius $r$ of the central coin will 
increase and $\theta_1$, the angle between the first and second coins, will 
decrease. Figure~\ref{90_180} illustrates this situation.
\begin{figure}[tb]
\centering
\includegraphics[scale=1]{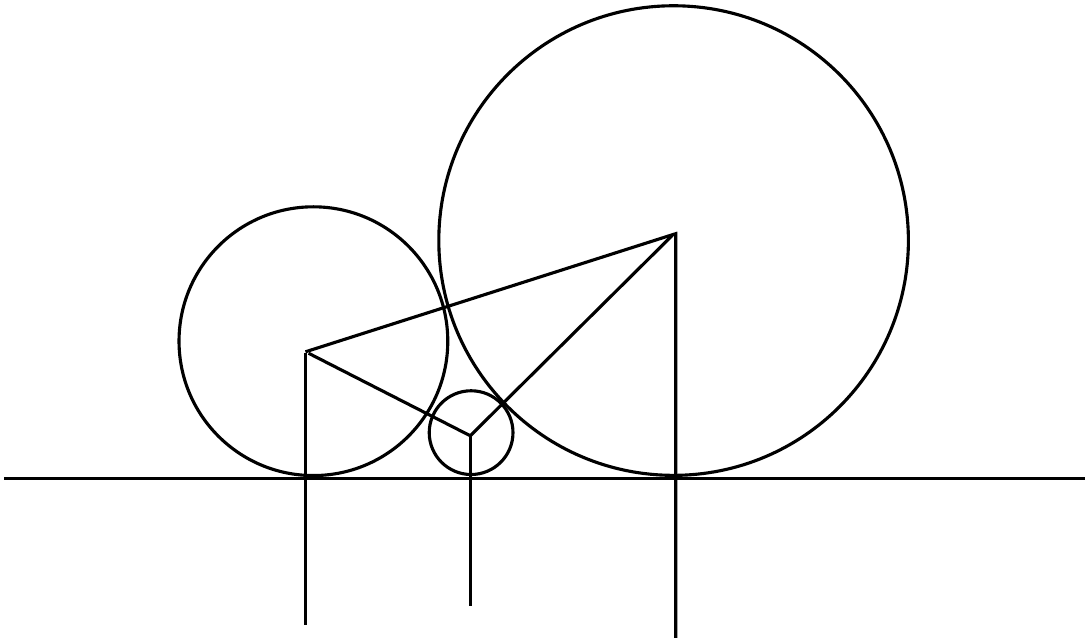}
\caption{A 3-petaled flower where the radius of one petal is 
increased to infinity.}
\label{90_180}
\end{figure}
It suffices to show that $\theta_1 > 90^{\circ}$ for this case. If we start 
with Figure~\ref{90_180} and draw a line parallel to the infinite circle 
that goes through the center of the central coin, we have 2 right triangles 
with side lengths $r_i-r, r_i + r$ and, by the Pythagorean theorem, 
$2 \sqrt{r_i r}$ for each $i= 1,2$. Therefore, the length of the 
segment forming the bottom of the rhombus, formed by the center
of the two outer circles and their touching points to the infinite
circle, is $2\left( \sqrt{r_1 r} + \sqrt{r_2 r} \right)$.  
We can now draw a segment parallel to this segment and passing through the 
center of the coin with the smaller radius.  Without loss of generality
we may assume $r_1 \leq r_2$. Now we have a right triangle with side lengths 
$2\left( \sqrt{r_1 r} + \sqrt{r_2 r} \right), r_2-r_1$ and $r_1 + r_2$
and hence by the Pythagorean theorem we have
$4\left( \sqrt{r_1 r} + \sqrt{r_2 r} \right)^2 + (r_2-r_1)^2= (r_1 + r_2)^2$,
which can be solved for $r$, obtaining 
\[
r = \frac{r_1 r_2}{\left(\sqrt{r_1} + \sqrt{r_2} \right)^2}.
\]
With this expression for $r$, it suffices to show that
$(r_1 + r_2)^2 > (r+r_1)^2 + (r+r_2)^2$, 
which implies $\theta_1 > 90^{\circ}$:
\begin{eqnarray*}
&   & (r_1 + r_2)^2 \\	
& = & \frac{(r_1 + r_2)^2\left(\sqrt{r_1} 
+ \sqrt{r_2} \right)^4}{\left(\sqrt{r_1} + \sqrt{r_2} \right)^4}\\
& = & \frac{r_1^4+r_2^4+
\left(14r_1^2 r_2^2+8r_1^3 r_2+12 r_1^{5/2} r_2^{3/2}+
12r_1^{3/2} r_2^{5/2}+8r_1 r_2^3\right) 
+4\left( r_1^{7/2} \sqrt{r_2}+\sqrt{r_1} r_2^{7/2}\right)}{\left(\sqrt{r_1} 
+ \sqrt{r_2} \right)^4}\\
& > & \frac{r_1^4+r_2^4+
8\left(r_1^2 r_2^2+ r_1^3 r_2+ r_1^{5/2} r_2^{3/2}+ 
r_1^{3/2} r_2^{5/2}+ r_1 r_2^3\right)
+4 \left( r_1^{7/2} \sqrt{r_2}+ \sqrt{r_1}r_2^{7/2}\right) }{\left(\sqrt{r_1} 
+ \sqrt{r_2} \right)^4}\\
& = & \left( \frac{r_1 r_2}{\left(\sqrt{r_1} + \sqrt{r_2} \right)^2} 
+ r_1\right)^2 + \left( \frac{r_1 r_2}{\left(\sqrt{r_1} + \sqrt{r_2} \right)^2} 
+ r_2\right)^2\\
& = & (r+r_1)^2 + (r+r_2)^2.
\end{eqnarray*}
\end{proof}
By Observation~\ref{obs:90180} we now know that for all the angles 
$\theta_i$, we have $90^{\circ} < \theta_i < 180^{\circ}$, and hence 
$-1 < \cos\theta_i < 0$. So in the parameterization of $x_1$ and $x_2$
\[
x_1 =  \frac{m_1^2-n_1^2}{m_1^2+ n_1^2}, \ \ 
x_2 =  \frac{m_2^2-n_2^2}{m_2^2+ n_2^2},
\]
we must choose $n_i > m_i$.  In this case $x_3$ in (\ref{eqn:x3})
must satisfy 
$\left(m_1^2-n_1^2\right)\left(m_2^2-n_2^2\right) -  4 m_1 m_2 n_1 n_2<0$, 
which is equivalent to $(m_1 n_2 +  m_2 n_1)^2>(m_1 m_2 - n_1 n_2)^2$.
Since $m_i < n_i$ this is equivalent to 
$m_1 n_2 + m_2	n_1 > n_1 n_2 - m_1 m_2$, 
or equivalently 
\begin{equation}
\label{eqn:x3-cnstrt}
m_1 n_2 +m_2 n_1 + m_1 m_2 > n_1 n_2.
\end{equation}
Looking at the expression for $r_1$, 
\[
r_1 = \frac{n_1(m_1 n_2 + m_2 n_1)}{ - m_1^2 n_2- n_1^2 n_2 
\pm (m_1 n_1 n_2  +m_2 n_1^2 )},
\]
we see that in order for $r_1 > 0$ to hold we must
have $n_1(m_1 n_2+m_2 n_1 )>n_2(m_1^2+ n_1^2)$.  
Re-solving for the radii $r_2$ and $r_3$ using 
the positive term in the expression for $r_1$, we obtain:
\[
r_2= \frac{n_1 n_2}{   m_2 n_1+ m_1n_2 - n_1 n_2}, \ \ 
r_3= \frac{n_2(m_1 n_2 +m_2 n_1 )}{ n_1 n_2^2 + m_2^2 n_1 - m_1 n_2^2 
-m_2 n_1n_2  },
\]
which give us two additional constraints in order to assure positive radii: 
$m_2 n_1+ m_1 n_2 > n_1 n_2$ and $n_1(m_2^2+ n_2^2)> n_2(m_1 n_2 +m_2 n_1)$.  
Note that the first of these constraints is stronger than 
(\ref{eqn:x3-cnstrt}), and so will replace it in the following summarizing 
theorem:
\begin{theorem}
\label{thm:param}
Let $m_1, n_1, m_2, n_2 \in \mathbb{N}$ that satisfy 
$n_1>m_1$, 
$n_2>m_2$,  
$m_1 n_2 + m_2 n_1 > n_1 n_2$, 
$n_1 (m_1 n_2 + m_2 n_1)> n_2 (m_1^2 + n_1 ^2)$, 
and $n_1 (m_2^2 + n_2^2)> n_2 (m_1 n_2 + m_2 n_1 )$.  
Then all rational cosines $x_i$ of a 3-petal flower 
are parametrized by:
\[
x_1=\frac{m_1^2-n_1^2}{m_1^2+ n_1^2}, \ \
x_2=\frac{m_2^2-n_2^2}{m_2^2+ n_2^2}, \ \ 
x_3=\frac{\left(m_1^2-n_1^2\right)\left(m_2^2-n_2^2\right) 
-  4 m_1 m_2 n_1 n_2}{(m_1^2+n_1^2)(m_2^2+n_2^2)}.
\]
Assuming the center coin has radius one, then all the
rational radii $r_i$ of the outer coins
are paramtetrized by:
\begin{eqnarray*}
r_1 & = & \frac{n_1(m_1 n_2 + m_2 n_1)}{  m_1 n_1 n_2  +m_2 n_1^2 
- m_1^2 n_2- n_1^2 n_2 }\\
r_2	& = & \frac{n_1 n_2}{   m_2 n_1+ m_1 n_2 - n_1 n_2}\\
r_3 & = & \frac{n_2(m_1n_2 +m_2n_1 )}{ n_1 n_2^2 + m_2^2n_1 - m_1n_2^2  
-m_2n_1n_2   }.
\end{eqnarray*}
This parameterization characterizes all sets of 
four mutually tangent Soddy circles of rational radius in the plane.
\end{theorem}
{\sc Example:}
Consider $m_1=1$, $n_1=2$, $m_2=4$, and $n_2=5$.  
We can see that the constraints will be satisfied, in particular
the nontrivial ones 
$m_1 n_2 + m_2 n_1 = 1 \cdot 5 + 4 \cdot 2 = 13>10 = 2 \cdot 5 = n_1 n_2$,
$n_1 (m_1 n_2 + m_2 n_1) = 2 (1 \cdot 5 + 4 \cdot 2) = 26 
> 25=  5 (1 + 4) =  n_2 (m_1^2 + n_1 ^2)$ and 
$n_1 (m_2^2 + n_2^2)= 2 ( 16 + 25) = 82 >65 
= 5(1 \cdot 5 + 4 \cdot 2) =  n_2 (m_1 n_2 +  m_2 n_1)$.
Then we have $x_1 = -\frac{3}{ 5}$, $x_2 = -\frac{9}{ 41}$,
$x_3= -\frac{133}{ 205}$, and the corresponding radii
$r_1  = 26$ ,$r_2  =  \frac{54}{11}$, $r_3= \frac{ 351}{ 59}$.
By scaling by the factor of $\gcd(11,59) = 649$ we obtain
an integral flower with center radius of $r= 649$ and
the outer radii $r_1  =  16874$, $r_2  =  3186$ and $r_3=  3861$.

\subsection{Descartes' circle theorem and another parametrization}

A nice relation connecting the radii of four mutually
tangent Soddy circles in the Euclidean plane is given
by Descartes' circle theorem~\cite{Austin}.
\begin{theorem}[Descartes]  
A collection of four mutually tangent circles in the plane, where 
$b_i = 1/r_i$ denotes the curvatures of the circles, satisfies 
the relation
\[
b_1^2 + b_2^2 + b_3^2 + b_4^2=\frac{1}{2} (b_1 + b_2 + b_3 + b_4)^2.
\]
\end{theorem}
Four mutually tangent circles in the plane are many times 
referred to as \emph{Soddy circles} for Frederick Soddy, an English 
chemist who rediscovered Descartes' Circle Theorem in 1936~\cite{Austin}.
This theorem has also been generalized to higher dimensions.

It is straightforward to check that our rational parameterization from 
Theorem~\ref{thm:param} satisfies Descartes' circle theorem.
Another elegant parametrization of integer Soddy circles are
given by Graham et al.~in~\cite{Graham} in the following
theorem.
\begin{theorem}[Graham et al.]
\label{thm:graham}
The following parametrization characterizes the integral 
curvatures of a set of Soddy circles:
\[
b_1=x,\ \ b_2=d_1-x,\ \ b_3= d_2-x,\ \ b_4= -2m + d_1 + d_2 -x,
\]
where $x^2 + m^2 = d_1 d_2$ and $0 \leq 2m \leq d_1 \leq d_2$. 
\end{theorem}
We conclude this section by briefly comparing our rational parametrization
to the one given by Theorem~\ref{thm:graham}.  Suppose we have 
a 3-petal flower, the coins of which have integer radii.  Further, assume
the center coin is the first one with curvature $b_1$. By scaling
to make the center coin of radius one and conveniently permuting indices, 
the remaing outer coins have radii $r_1,r_2,r_3$ given by
\[
r_1= \frac{b_1}{b_2},\ \ r_2 = \frac{b_1}{b_4},\ \ r_3=\frac{b_1}{b_3}.
\]
By Theorem~\ref{thm:param} we have that
\begin{eqnarray*}
\frac{b_1}{b_2} & = & \frac{n_1( m_1n_2 + m_2 n_1)}{ -m_1^2 n_2  
- n_1^2 n_2 + m_1 n_1 n_2  +m_2 n_1^2 }\\
\frac{b_1}{b_3} & = & \frac{n_2(m_1n_2 +m_2 n_1)}{ - m_1n_2^2  
-m_2n_1n_2   +n_1n_2^2 +  m_2^2n_1}\\
\frac{b_1}{b_4} & = & \frac{n_1 n_2}{ - n_1 n_2 + m_2 n_1
+ m_1n_2 }.
\end{eqnarray*}
Replacing each $b_i$ with the integer parametrization from 
Theoerm~\ref{thm:graham} we can solve for $d_1/x,d_2/x$ and $m/x$
in terms of $m_1, m_2, n_1, n_2,$ and obtain 
\[
\frac{m}{x} = \frac{(n_1 n_2 - m_1 m_2)}{m_1n_2 + m_2n_1},\ \ 
\frac{d_1}{x} =\frac{n_2 (m_1^2 + n_1^2)}{n_1(m_1n_2 + m_2 n_1)}, \ \ 
\frac{d_2}{x} =\frac{n_1(m_2^2 + n_2^2)}{n_2(m_1 n_2 + m_2 n_1)}.
\]
Hence, $1 +(m/x)^2 = (d_1/x)(d_2/x)$ so the quadratic equation
relating the parameters in Theorem~\ref{thm:graham} is satisfied.

The first inequality, $0 \leq 2m$, will clearly holds
since  when we choose $m_i <n_i$ for $i=1,2$.

Using the inequality constraints from Theorem~\ref{thm:param} we obtain
\[
\frac{d_1}{x} = \frac{n_2 (m_1^2 + n_1^2)}{n_1(m_1n_2 + m_2 n_1)}
\leq \frac{n_1 n_2(m_1n_2 + m_2n_1)}{n_1n_2(m_1n_2 + m_2 n_1)}
\leq \frac{n_1(m_2^2 + n_2^2)}{n_2(m_1 n_2 + m_2 n_1)}
= \frac{d_2}{x},
\]
and hence third inequality $d_1 \leq d_2$ holds.  

However, the second inequality, $2m \leq d_1$ in 
Theorem~\ref{thm:graham}, does not hold.
In fact, one can show that the {\em opposite} inequality holds 
for all $m_1, m_2, n_1, n_2,$ that satisfy 
the conditional inequalities given in Theorem~\ref{thm:param}.
This does not mean there is anything wrong with the parametrization
in either Theorem~\ref{thm:param} or Theorem~\ref{thm:graham},
since different range for parameters certainly can yield same 
solution set. Insisting the opposite $d_1\leq 2m$ in Theorem~\ref{thm:graham}
might also yield all integer curvatures of Soddy circles.

Although equivalent, there is a subtle difference between presenting
the integer radii of Soddy circles and presenting the integer
curvatures. Suppose we
have integer curvatures $b_1,b_2,b_3$ and $b_4$ of Soddy circles,
and we would like to find the corresponding scaled configuration
of Soddy circles with integer radii. Hence we are seeking 
$r_1,r_2,r_3,r_4$ and $N$ such that $N/r_i = b_i$ for each
$i$. As $\lcm(b_1,b_2,b_3,b_4)$ divides $N$ we have that 
$r_i = k\cdot\lcm(b_1,b_2,b_3,b_4)/b_i$ for each $i$, where $k$
is some positive integer. The other conversion, from integer
radii to integer curvatures is similar.

In conclusion, we see that Graham et al.'s characterization of integer
curvatures of Soddy circles in Theorem~\ref{thm:graham} is implied by 
our rational parametrization of the radii of 3-petal flowers
in Theorem~\ref{thm:param}. In addition, the parametrization given 
by Graham et al in Theorem~\ref{thm:graham} 
relies on solving the Diophantine equation $x^2 + m^2 = d_1 d_2$ for
each chosen $x$, $d_1$, and $d_2$, while 
the parametrization developed here and given in Theorem~\ref{thm:param} 
does not rely on satisfying any such equation, only inequalities. 

\subsection*{Acknowledgments}  

Sincere thanks to the anonymous referees for ...

\flushright{\today}

\flushleft

\appendix

\section{Generalizations of the Pythagorean Triples}
\label{appx:beta}

In the following, a {\em primitive solution} is a solution 
where $x, y,$ and $z$ are pairwise relatively prime.
To prove Theorem~\ref{thm:pyth_ext}, we need the following:
\begin{claim}
\label{claim:rel-prime}
If r,s,t are positive integers such that r and s are relatively prime and 
$rs=t^2$ then there are relatively prime integers m and n such that $r=m^2$ 
and $s=n^2$.
\end{claim}
\begin{proof}[Theorem \ref{thm:pyth_ext}]
Note that $\gcd(b, c)=1$.  This proof follows and extends the exposition 
in~\cite{Rosen}.  

Assume $x, y, z$ form a primitive solution.  In this case, $x$ and $y$ 
cannot both be even.

Case 1: $x,y$ are both odd.  Then $x^2 \equiv 1 \pmod 4$ and 
$y^2 \equiv 1 \pmod 4$, giving $z^2 \equiv 1+ \beta \pmod 4$.  
Since $z^2 \equiv 0,1 \pmod 4$, then $\beta \equiv 0$ or 
$\beta \equiv 3 \pmod 4$ must hold.  However, $\beta \equiv 0 \pmod 4$ 
implies that $4$ divides $\beta$, contradicting the assumption that 
$\beta$ is square-free.  So the only case to consider here is the case 
where $z$ is even and $\beta \equiv 3 \pmod 4$.
\begin{eqnarray}\label{factor_beta}
\beta y^2	& = & z^2 - x^2 =  (z+x)(z-x).
\end{eqnarray}
Letting $\gcd(z+x,z-x) = d$ we get that $d$ divides both $z+x+z-x=2z$ and 
$z+x-(z-x)= 2x$.  Since $x$ and $z$ are relatively prime, $d=1$ or $2$.  
Since both $z+x$ and $z-x$ are odd, then $d=1$ must hold.  Since now 
$\gcd(z+x, z-x)=1$ we have from (\ref{factor_beta}) that for some 
factorization $\beta=bc$ then $r=z+x$ is divisible by $b$ and $s=z-x$ 
is divisible by $c$.  Since $\gcd \left(\frac{r}{b}, \frac{s}{c}\right) = 1$, 
we have by Claim \ref{claim:rel-prime} that $m^2 = \frac{r}{b}$ and 
$n^2 = \frac{s}{c}$, and hence $y=mn$, 
$x= \frac{r-s}{2}= \frac{b m^2- c n^2}{2}$, and 
$z= \frac{r+s}{2}= \frac{b m^2+ c n^2}{2}$.

Case 2: $x$ is even and $y$ is odd.  Then $x^2 \equiv 0 \pmod 4$ and 
$y^2 \equiv 1 \pmod 4$, giving $z^2 \equiv \beta \pmod 4$.  Therefore 
$\beta \equiv 0$ or $\beta \equiv 1 \pmod 4$.  However, $
\beta \equiv 0 \pmod 4$ implies that $4$ divides $\beta$, 
again contradicting the assumption that $\beta$ is square-free.  
So the only case to consider here is the case where $z$ is odd and 
$\beta \equiv 1 \pmod 4$, which proceeds exactly as in case 1.

Case 3: $x$ is odd and $y$ is even.  Then $x^2 \equiv 1 \pmod 4$ and 
$y^2 \equiv 0 \pmod 4$, giving $z^2 \equiv 1 \pmod 4$, and so $z$ is odd.

Unlike cases 1 and 2, $z+x$ and $z-x$ are both even.  Letting 
$\gcd \left(\frac{z+x}{2},\frac{z-x}{2}\right) = d$ we get that 
$d$ divides $\frac{z+x+z-x}{2}=z$ and $\frac{z+x-(z-x)}{2}= x$.  
Since $x$ and $z$ are relatively prime, $d=1$.  Now we have 
$\frac{\beta y^2}{4} = rs$ where $r=\frac{z+x}{2}$ and 
$s=\frac{z-x}{2}$.  Hence $b$ divides $r$ and $c$ divides $s$ for 
some appropriate factorization $\beta= b c$. Since 
$\gcd \left(\frac{r}{b}, \frac{s}{c}\right) = 1$, so we have by 
Claim \ref{claim:rel-prime} that $m^2 = \frac{r}{b}$ and 
$n^2 = \frac{s}{c}$ and hence $y=2mn$, $x= r-s= b m^2- c n^2$, and 
$z= r+s= b m^2+ c n^2$.

For the other direction, first we show that $x, y, z$ as given in 
cases 1 and 2 do form a solution:
\begin{eqnarray*}
x^2 + \beta y^2 	
& = & \left(\frac{b m^2-c n^2}{2}\right)^2 + \beta (mn)^2\\			
& = & \frac{(b m^2)^2- 2 b m^2 c n^2+ (c n^2)^2}{4} + \beta(mn)^2\\	
& = & \frac{(b m^2)^2+ 2 \beta m^2 n^2+ (c n^2)^2}{4} \\
& = & \left(\frac{b m^2+c n^2}{2}\right)^2.
\end{eqnarray*}
Also for case 3 we get:
\begin{eqnarray*}
x^2 + \beta y^2 	
& = & \left(b m^2-c n^2\right)^2 + \beta(2mn)^2	\\
& = & (b m^2)^2- 2 b m^2 c n^2+ (c n^2)^2+ \beta(2mn)^2\\
& = & (b m^2)^2+ 2 \beta m^2 n^2+ (c n^2)^2 \\
& = & \left(b m^2+c n^2\right)^2.
\end{eqnarray*}

To show that the triple is primitive for cases 1 and 2, assume on the 
contrary that $\gcd(x,y,z)=d>1$.  Then there is a prime $p$ that divides 
$d$.  This $p$ divides $x$ and $z$ and also their sum and difference: 
$x+z= \frac{b m^2- c n^2}{2} + \frac{b m^2+ c n^2}{2} = b m^2$ and
$x-z= \frac{b m^2- c n^2}{2} - \frac{b m^2+ c n^2}{2} = c n^2$.  This 
contradicts the assumption that $b m^2$ and $c n^2$ are relatively prime.

For case 3, again assume on the contrary that $(x,y,z)=d>1$.  Then there 
is an odd prime $p$ that divides $d$.  $p \neq 2$ because $x$ and $z$ are 
both odd.  This $p$ divides $x$ and $z$ and also their sum and difference: 
$x+z = 2 b m^2$ and $x-z = 2 c n^2$.  Again, this contradicts the assumption 
that $b m^2$ and $c n^2$ are relatively prime.
\end{proof}

\section{The polynomial $P_5$}
\label{appx:P5}

\begin{eqnarray*}
P_5	
& = & P_4(x_1, x_2, x_3, \EC_2(x_4, x_5)) 
\cdot P_4(x_1, x_2, x_3, \overline{\EC_2}(x_4, x_5))\\		
& = & x_5^8-8 x_1 x_2 x_3 x_4 x_5^7-8 x_3^2 x_4^2 x_5^6+
4 x_2^2 x_5^6-4 x_5^6+4 x_3^2 x_5^6+16 x_1^2 x_2^2 x_3^2 x_5^6\\		
& - & 8 x_2^2 x_3^2 x_5^6-8 x_1^2 x_4^2 x_5^6-8 x_1^2 x_3^2 x_5^6-
8 x_1^2 x_2^2 x_5^6+4 x_4^2 x_5^6-8 x_2^2 x_4^2 x_5^6\\
& + & 16 x_1^2 x_3^2 x_4^2 x_5^6+16 x_2^2 x_3^2 x_4^2 x_5^6
+4 x_1^2 x_5^6+16 x_1^2 x_2^2 x_4^2 x_5^6+40 x_1 x_2^3 x_3 x_4 x_5^5\\
& + & 40 x_1 x_2 x_3 x_4^3 x_5^5-32 x_1^3 x_2 x_3 x_4^3 x_5^5
+40 x_1^3 x_2 x_3 x_4 x_5^5-32 x_1^3 x_2 x_3^3 x_4 x_5^5\\		
& - & 32 x_1 x_2^3 x_3^3 x_4 x_5^5-32 x_1 x_2^3 x_3 x_4^3 x_5^5
-24 x_1 x_2 x_3 x_4 x_5^5-32 x_1^3 x_2^3 x_3 x_4 x_5^5\\
& - & 32 x_1 x_2 x_3^3 x_4^3 x_5^5+40 x_1 x_2 x_3^3 x_4 x_5^5
+64 x_1^2 x_2^4 x_3^2 x_4^2 x_5^4-16 x_1^4 x_4^2 x_5^4\\
& + & 28 x_2^2 x_4^2 x_5^4-16 x_3^2 x_4^4 x_5^4
-24 x_1^2 x_2^2 x_3^2 x_5^4+28 x_1^2 x_4^2 x_5^4-12 x_3^2 x_5^4\\
& + & 28 x_2^2 x_3^2 x_5^4-16 x_2^2 x_3^4 x_5^4-16 x_2^2 x_4^4 x_5^4
+64 x_1^2 x_2^2 x_3^4 x_4^2 x_5^4-24 x_2^2 x_3^2 x_4^2 x_5^4\\
& + & 16 x_1^4 x_4^4 x_5^4-12 x_4^2 x_5^4-24 x_1^2 x_3^2 x_4^2 x_5^4
+6 x_5^4+6 x_2^4 x_5^4-16 x_1^2 x_3^4 x_5^4\\
& - & 16 x_1^2 x_4^4 x_5^4+6 x_4^4 x_5^4+64 x_1^4 x_2^2 x_3^2 x_4^2 x_5^4
+16 x_1^4 x_3^4 x_5^4+16 x_2^4 x_3^4 x_5^4\\
& - & 16 x_1^4 x_2^2 x_5^4-24 x_1^2 x_2^2 x_4^2 x_5^4+16 x_1^4 x_2^4 x_5^4
+16 x_3^4 x_4^4 x_5^4+6 x_3^4 x_5^4-12 x_2^2 x_5^4\\
& - & 144 x_1^2 x_2^2 x_3^2 x_4^2 x_5^4+64 x_1^2 x_2^2 x_3^2 x_4^4 x_5^4
+28 x_1^2 x_3^2 x_5^4-16 x_2^4 x_4^2 x_5^4\\
& - & 16 x_2^4 x_3^2 x_5^4-16 x_3^4 x_4^2 x_5^4+28 x_3^2 x_4^2 x_5^4
+28 x_1^2 x_2^2 x_5^4-12 x_1^2 x_5^4-16 x_1^4 x_3^2 x_5^4\\
& + & 6 x_1^4 x_5^4-16 x_1^2 x_2^4 x_5^4+16 x_2^4 x_4^4 x_5^4
+112 x_1^3 x_2 x_3^3 x_4 x_5^3+112 x_1^3 x_2 x_3 x_4^3 x_5^3\\
& + & 40 x_1 x_2 x_3^5 x_4 x_5^3-32 x_1^5 x_2 x_3 x_4^3 x_5^3
-32 x_1 x_2^5 x_3 x_4^3 x_5^3+112 x_1^3 x_2^3 x_3 x_4 x_5^3\\
& + & x_1 x_2^5 x_3 x_4 x_5^3+40 x_1^5 x_2 x_3 x_4 x_5^3
+40 x_1 x_2 x_3 x_4^5 x_5^3-32 x_1^3 x_2 x_3 x_4^5 x_5^3\\
& - & 112 x_1 x_2^3 x_3 x_4 x_5^3-32 x_1 x_2 x_3^5 x_4^3 x_5^3
-32 x_1^5 x_2 x_3^3 x_4 x_5^3-112 x_1^3 x_2 x_3 x_4 x_5^3\\
& + & 112 x_1 x_2^3 x_3^3 x_4 x_5^3-32 x_1 x_2^3 x_3^5 x_4 x_5^3
-32 x_1 x_2^5 x_3^3 x_4 x_5^3-32 x_1 x_2^3 x_3 x_4^5 x_5^3\\
& - & 32 x_1^3 x_2^5 x_3 x_4 x_5^3+112 x_1 x_2 x_3^3 x_4^3 x_5^3
+112 x_1 x_2^3 x_3 x_4^3 x_5^3-112 x_1 x_2 x_3^3 x_4 x_5^3\\
& - & 112 x_1 x_2 x_3 x_4^3 x_5^3+72 x_1 x_2 x_3 x_4 x_5^3
-128 x_1^3 x_2^3 x_3^3 x_4^3 x_5^3-32 x_1^5 x_2^3 x_3 x_4 x_5^3\\
& - & 32 x_1^3 x_2 x_3^5 x_4 x_5^3-32 x_1 x_2 x_3^3 x_4^5 x_5^3
+16 x_1^2 x_2^6 x_4^2 x_5^2+28 x_3^2 x_4^4 x_5^2+28 x_1^4 x_3^2 x_5^2\\
& - & 8 x_2^2 x_4^6 x_5^2-12 x_4^4 x_5^2-16 x_1^4 x_4^4 x_5^2
+16 x_1^2 x_3^2 x_4^6 x_5^2-24 x_1^2 x_2^2 x_3^4 x_5^2-8 x_1^6 x_3^2 x_5^2\\
& - & 32 x_2^2 x_3^2 x_5^2-24 x_2^2 x_3^2 x_4^4 x_5^2
-16 x_1^4 x_2^4 x_5^2-8 x_1^6 x_4^2 x_5^2+40 x_2^2 x_3^2 x_4^2 x_5^2\\
& - & 16 x_1^4 x_3^4 x_5^2-24 x_1^4 x_2^2 x_4^2 x_5^2
+28 x_3^4 x_4^2 x_5^2-8 x_1^2 x_2^6 x_5^2+64 x_1^2 x_2^4 x_3^4 x_4^2 x_5^2\\
& - & 32 x_1^2 x_4^2 x_5^2-16 x_3^4 x_4^4 x_5^2-12 x_3^4 x_5^2
+28 x_2^2 x_3^4 x_5^2-144 x_1^2 x_2^2 x_3^4 x_4^2 x_5^2\\
& + & 16 x_2^2 x_3^6 x_4^2 x_5^2+28 x_1^2 x_4^4 x_5^2+4 x_4^6 x_5^2
+64 x_1^2 x_2^2 x_3^4 x_4^4 x_5^2-24 x_2^4 x_3^2 x_4^2 x_5^2\\
& + & 28 x_2^4 x_3^2 x_5^2-8 x_1^6 x_2^2 x_5^2+16 x_1^6 x_2^2 x_3^2 x_5^2
+16 x_1^2 x_2^2 x_4^6 x_5^2+16 x_2^2 x_3^2 x_4^6 x_5^2\\
& + & 12 x_4^2 x_5^2-24 x_1^4 x_2^2 x_3^2 x_5^2
+16 x_1^2 x_3^6 x_4^2 x_5^2-24 x_1^2 x_2^4 x_3^2 x_5^2-8 x_2^2 x_3^6 x_5^2\\
& + & 64 x_1^4 x_2^2 x_3^2 x_4^4 x_5^2+192 x_1^2 x_2^2 x_3^2 x_4^2 x_5^2
-12 x_2^4 x_5^2-24 x_1^2 x_2^4 x_4^2 x_5^2+12 x_2^2 x_5^2\\
& - & 8 x_1^2 x_3^6 x_5^2+40 x_1^2 x_3^2 x_4^2 x_5^2
-24 x_1^2 x_2^2 x_4^4 x_5^2-32 x_1^2 x_2^2 x_5^2+64 x_1^4 x_2^2 x_3^4 x_4^2 x_5^2\\
& + & 28 x_2^2 x_4^4 x_5^2-8 x_1^2 x_4^6 x_5^2-4 x_5^2
+4 x_1^6 x_5^2+12 x_1^2 x_5^2+28 x_1^4 x_4^2 x_5^2-16 x_2^4 x_3^4 x_5^2\\
& + & 16 x_1^6 x_3^2 x_4^2 x_5^2-8 x_2^6 x_4^2 x_5^2
+16 x_1^6 x_2^2 x_4^2 x_5^2+64 x_1^4 x_2^4 x_3^2 x_4^2 x_5^2\\
& - & 24 x_1^2 x_3^2 x_4^4 x_5^2+12 x_3^2 x_5^2+16 x_1^2 x_2^2 x_3^6 x_5^2
+16 x_2^6 x_3^2 x_4^2 x_5^2+16 x_1^2 x_2^6 x_3^2 x_5^2\\
& + & 40 x_1^2 x_2^2 x_4^2 x_5^2-8 x_3^2 x_4^6 x_5^2
-24 x_1^2 x_3^4 x_4^2 x_5^2-16 x_2^4 x_4^4 x_5^2+28 x_1^2 x_3^4 x_5^2\\
\end{eqnarray*}
\begin{eqnarray*}
& - & 144 x_1^2 x_2^2 x_3^2 x_4^4 x_5^2+28 x_2^4 x_4^2 x_5^2
-8 x_3^6 x_4^2 x_5^2-32 x_2^2 x_4^2 x_5^2+64 x_1^2 x_2^4 x_3^2 x_4^4 x_5^2\\
& + & 40 x_1^2 x_2^2 x_3^2 x_5^2-24 x_2^2 x_3^4 x_4^2 x_5^2
-144 x_1^2 x_2^4 x_3^2 x_4^2 x_5^2-12 x_1^4 x_5^2+4 x_3^6 x_5^2\\
& + & 28 x_1^2 x_2^4 x_5^2-144 x_1^4 x_2^2 x_3^2 x_4^2 x_5^2
-8 x_2^6 x_3^2 x_5^2+28 x_1^4 x_2^2 x_5^2-32 x_3^2 x_4^2 x_5^2\\
& - & 32 x_1^2 x_3^2 x_5^2+4 x_2^6 x_5^2-24 x_1^4 x_3^2 x_4^2 x_5^2
-24 x_1 x_2^5 x_3 x_4 x_5-112 x_1^3 x_2 x_3 x_4^3 x_5\\
& - & 112 x_1^3 x_2^3 x_3 x_4 x_5+40 x_1^3 x_2 x_3^5 x_4 x_5
-112 x_1 x_2 x_3^3 x_4^3 x_5-32 x_1 x_2^3 x_3^5 x_4^3 x_5\\
& - & 32 x_1^3 x_2 x_3^5 x_4^3 x_5-8 x_1 x_2 x_3^7 x_4 x_5
+40 x_1^3 x_2 x_3 x_4^5 x_5-24 x_1 x_2 x_3 x_4^5 x_5\\
& - & 112 x_1 x_2^3 x_3 x_4^3 x_5+40 x_1 x_2^5 x_3^3 x_4 x_5
+40 x_1 x_2^3 x_3 x_4^5 x_5-32 x_1^5 x_2 x_3^3 x_4^3 x_5\\
& - & 32 x_1^3 x_2^3 x_3^5 x_4 x_5+40 x_1^3 x_2^5 x_3 x_4 x_5
-8 x_1 x_2^7 x_3 x_4 x_5-8 x_1 x_2 x_3 x_4^7 x_5+112 x_1 x_2^3 x_3^3 x_4^3 x_5\\
& + & 40 x_1 x_2^5 x_3 x_4^3 x_5-112 x_1^3 x_2 x_3^3 x_4 x_5
+40 x_1^5 x_2 x_3^3 x_4 x_5+72 x_1 x_2 x_3^3 x_4 x_5\\
& - & 32 x_1 x_2^3 x_3^3 x_4^5 x_5+72 x_1 x_2^3 x_3 x_4 x_5
+40 x_1^5 x_2^3 x_3 x_4 x_5+40 x_1^5 x_2 x_3 x_4^3 x_5\\
& - & 32 x_1^3 x_2^3 x_3 x_4^5 x_5-24 x_1 x_2 x_3^5 x_4 x_5
-32 x_1^5 x_2^3 x_3 x_4^3 x_5-32 x_1^3 x_2^5 x_3 x_4^3 x_5\\
& - & 32 x_1 x_2^5 x_3^3 x_4^3 x_5+112 x_1^3 x_2^3 x_3 x_4^3 x_5
+40 x_1 x_2 x_3^5 x_4^3 x_5-40 x_1 x_2 x_3 x_4 x_5\\
& + & 40 x_1 x_2 x_3^3 x_4^5 x_5-32 x_1^5 x_2^3 x_3^3 x_4 x_5
-112 x_1 x_2^3 x_3^3 x_4 x_5-24 x_1^5 x_2 x_3 x_4 x_5\\
& - & 32 x_1^3 x_2^5 x_3^3 x_4 x_5+72 x_1 x_2 x_3 x_4^3 x_5
-8 x_1^7 x_2 x_3 x_4 x_5+112 x_1^3 x_2^3 x_3^3 x_4 x_5\\
& + & 112 x_1^3 x_2 x_3^3 x_4^3 x_5-32 x_1^3 x_2 x_3^3 x_4^5 x_5
+40 x_1 x_2^3 x_3^5 x_4 x_5+72 x_1^3 x_2 x_3 x_4 x_5\\
& + & 28 x_1^4 x_3^2 x_4^2+16 x_1^2 x_2^6 x_3^2 x_4^2
-24 x_1^2 x_2^2 x_3^2 x_4^4+28 x_1^2 x_2^2 x_4^4-8 x_1^2 x_3^2 x_4^6\\
& - & 16 x_1^2 x_2^4 x_4^4+28 x_2^2 x_3^4 x_4^2+28 x_1^4 x_2^2 x_4^2
-32 x_1^2 x_2^2 x_4^2-8 x_1^2 x_2^2 x_4^6+4 x_3^2 x_4^6\\
& - & 24 x_1^2 x_2^2 x_3^4 x_4^2-12 x_3^2 x_4^4+16 x_2^4 x_3^4 x_4^4
+16 x_1^2 x_2^2 x_3^6 x_4^2+4 x_2^2 x_4^6-32 x_1^2 x_3^2 x_4^2\\
& + & 4 x_2^2 x_3^6-24 x_1^2 x_2^4 x_3^2 x_4^2-12 x_2^4 x_4^2
+28 x_1^2 x_3^4 x_4^2+6 x_1^4 x_4^4-16 x_1^4 x_3^2 x_4^4+28 x_1^2 x_2^4 x_4^2\\
& + & 12 x_1^2 x_3^2+6 x_3^4 x_4^4+16 x_1^4 x_2^4 x_4^4-8 x_2^2 x_3^2 x_4^6
-12 x_2^2 x_4^4-32 x_2^2 x_3^2 x_4^2-8 x_2^2 x_3^6 x_4^2\\
& + & 12 x_2^2 x_4^2+4 x_2^6 x_3^2-16 x_1^4 x_2^2 x_3^4-16 x_2^4 x_3^4 x_4^2
+40 x_1^2 x_2^2 x_3^2 x_4^2+6 x_3^4-8 x_2^6 x_3^2 x_4^2\\
& - & 4 x_2^2-16 x_1^2 x_3^4 x_4^4+x_3^8-12 x_1^2 x_4^4-16 x_2^4 x_3^2 x_4^4
+6 x_1^4 x_3^4-16 x_1^4 x_2^4 x_4^2-12 x_1^4 x_2^2\\
& + & 16 x_1^4 x_3^4 x_4^4-4 x_4^2-8 x_1^2 x_3^6 x_4^2+x_4^8-8 x_1^2 x_2^6 x_3^2
-16 x_1^4 x_3^4 x_4^2+28 x_1^2 x_2^2 x_3^4\\
& + & 4 x_3^6 x_4^2+16 x_1^4 x_2^4 x_3^4+16 x_1^6 x_2^2 x_3^2 x_4^2
+12 x_2^2 x_3^2+4 x_1^2 x_3^6+4 x_1^2 x_4^6+4 x_2^6 x_4^2\\
& - & 8 x_1^2 x_2^2 x_3^6+16 x_1^2 x_2^2 x_3^2 x_4^6-4 x_4^6-8 x_1^6 x_2^2 x_3^2
-12 x_1^2 x_2^4-16 x_1^4 x_2^4 x_3^2-12 x_1^2 x_3^4\\
& - & 12 x_2^4 x_3^2-16 x_2^2 x_3^4 x_4^4+12 x_1^2 x_4^2+x_1^8+4 x_1^6 x_4^2
-24 x_1^4 x_2^2 x_3^2 x_4^2-8 x_1^2 x_2^6 x_4^2\\
& + & 6 x_4^4+12 x_3^2 x_4^2+28 x_2^2 x_3^2 x_4^4+6 x_1^4 x_2^4+6 x_1^4
+28 x_1^4 x_2^2 x_3^2+28 x_2^4 x_3^2 x_4^2+6 x_2^4 x_3^4\\
& - & 32 x_1^2 x_2^2 x_3^2+4 x_1^2 x_2^6-4 x_3^2-4 x_1^6-4 x_1^2
-8 x_1^6 x_3^2 x_4^2+x_2^8-16 x_1^4 x_2^2 x_4^4-16 x_1^2 x_2^4 x_3^4\\
& + & 4 x_1^6 x_2^2+6 x_2^4 x_4^4-4 x_3^6-8 x_1^6 x_2^2 x_4^2
-12 x_3^4 x_4^2+12 x_1^2 x_2^2-12 x_2^2 x_3^4+28 x_1^2 x_3^2 x_4^4\\
& - & 12 x_1^4 x_4^2+28 x_1^2 x_2^4 x_3^2+1-4 x_2^6-12 x_1^4 x_3^2
+6 x_2^4+4 x_1^6 x_3^2.
\end{eqnarray*}


\begin{thebibliography}{10}


\bibitem{Andreev}
\newblock E.~M.~Andreev:
\newblock Convex polyhedra in Loba\v{c}evski\u{i} spaces,
\newblock {\em Matematicheski\u{i} Sbornik. Novaya Seriya},
\newblock \textbf{81}, no.~123: 445 -- 478, (1970).

\bibitem{Austin}
\newblock David Austin:
\newblock When Kissing Involves Trigonometry
\newblock {\em AMS Features Column}, {\bf 9}: (1999).

\bibitem{Brass}
\newblock Peter Brass; William Moser; Janos Pach:
\newblock {\em Research Problems in Discrete Geometry},
\newblock Springer-Verlag, New York, (2005).

\bibitem{Scheinerman}
\newblock Graham R. Brightwell; Edward R. Scheinerman: 
\newblock Representations of planar graphs,  
\newblock {\em SIAM Journal of Discrete Mathematics}, {\bf 6}: 214--229, (2000).

\bibitem{Graham}
\newblock R.L. Graham and J.C. Lagarias and C.L. Mallows and A. 
Wilks and C. Yan:
\newblock Apollonian circle packings: number theory,
\newblock {\em Journal of Number Theory}, {\bf 100}: 1--45, (2003).

\bibitem{Hungerford}
\newblock Thomas Hungerford:
\newblock {\em Algebra},
\newblock Graduate Texts in Mathematics, GTM-73 Springer-Verlag, (1974).

\bibitem{Koebe}
\newblock Paul Koebe:
\newblock Kontaktprobleme der konformen Abbildung,
\newblock {\em Ber.~Verh.~S\"{a}chs, Akademie der 
Wissenshaften Leipzig, Math.-Phys. Klasse}, {\bf 88}: 141 -- 164, (1936).

\bibitem{MAPLE}
\newblock MAPLE:
\newblock {\em mathematics software tool for symbolic computation},
\newblock {\tt http://www.maplesoft.com/products/Maple/academic/index.aspx}

\bibitem{Rosen}
\newblock Kenneth H. Rosen:
\newblock {\em Elementary Number Theory and Its Applications},
\newblock Pearson Addison Wesley, (2005).

\bibitem{Stanley}
\newblock Richard Stanley:
\newblock Invariants of finite groups and their applications to combinatorics, 
\newblock {\em Bull. Amer. Math. Soc.}, {\bf 3}: 475--511, (1979).

\bibitem{Stephenson}
\newblock Kenneth Stephenson:
\newblock {\em Introduction to Circle Packing : 
The Theory of Discrete Analytic Functions},
\newblock Cambridge University Press, (2005).

\bibitem{Thurston}
\newblock William Thurston:
\newblock Three-Dimensional Geometry and Topology,
\newblock {\em Princeton University Press}, (1997).

\bibitem{Ziegler}
\newblock G\"{u}nter M.~Ziegler:
\newblock {\em Lectures on Polytopes},
\newblock Gratuate Texts in Mathematics, GMT-152 Springer Verlag, (1995).

\end{thebibliography}
\end{document}